\documentclass[11pt,leqno]{amsart}
\allowdisplaybreaks[4]
\usepackage{amsmath,graphicx,amssymb,amscd}
\usepackage[initials]{amsrefs}
\usepackage{listings,jvlisting}
\usepackage{here}
\newtheorem{lem}{Lemma}
\newtheorem{thm}{Theorem}

\newtheorem{thA}{Theorem A}

\newtheorem{corollary}{Corollary}
\newtheorem{proposition}{Proposition}
\newtheorem{rem}{Remark}

\title{Uniqueness and multiple existence of positive radial solutions of the Brezis-Nirenberg Problem on annular domains in ${\Bbb S}^{3}$}
\author[N. Shioji]{Naoki Shioji${}^\dag$}
\author[S. Tanaka]{Satoshi Tanaka${}^*$}
\author[K. Watanabe]{Kohtaro Watanabe${}^\ddag$}
\email[Naoki Shioji]{shioji193@gmail.com}
\email[Satoshi Tanaka]{satoshi.tanaka.d4@tohoku.ac.jp}
\email[Kohtaro Watanabe]{wata@nda.ac.jp}
\address[Naoki Shioji]{Department of Mathematics, 
Faculty of Engineering,
Yokohama National University,
Tokiwadai, Hodogaya-ku,
Yokohama 240-8501, Japan}
\address[Satoshi Tanaka]{
Mathematical Institute, Tohoku University,
6-3-09 Aramaki-Aza-Aoba, Aoba-ku, Sendai 980-8579, Japan
}
\address[Kohtaro Watanabe]{
Department of Computer Science, National Defense Academy, 1-10-20
Hashirimizu, Yokosuka 239-8686, Japan}
\begin{document}
\begin{abstract}
The uniqueness and multiple existence of positive radial solutions to the Brezis-Nirenberg problem on a domain in the 3-dimensional unit sphere 
${\mathbb S}^3$
\begin{equation*}
\left\{
\begin{aligned}
\Delta_{{\mathbb S}^3}U -\lambda U + U^p&=0,\, U>0 && \text{in $\Omega_{\theta_1,\theta_2}$,}\\
U &= 0&&\text{on $\partial \Omega_{\theta_1,\theta_2}$,}
  \end{aligned}
\right.
\end{equation*}
for $-\lambda_{1}<\lambda\leq 1$ are shown, where $\Delta_{{\mathbb S}^3}$ is the Laplace-Beltrami operator, $\lambda_{1}$ is the first eigenvalue of $-\Delta_{{\mathbb S}^3}$ and  $\Omega_{\theta_1,\theta_2}$ is an annular domain in ${\mathbb S}^3$: whose great circle distance (geodesic distance)
from $(0,0,0,1)$ is greater than $\theta_1$ and less than $\theta_2$. A solution is said to be radial if it depends only on this geodesic distance. 
It is proved that the number of positive radial solutions of the problem changes with respect to the exponent $p$ and parameter $\lambda$
when $\theta_1=\varepsilon$, $\theta_2=\pi-\varepsilon$ and $0<\varepsilon$ is sufficiently small.
\end{abstract}
\maketitle
\section{Introduction}
We show the uniqueness and multiple existence of positive radial solutions to Brezis-Nirenberg type problem on 3-dimensional unit sphere ${\mathbb S}^3$:
 \begin{equation}\label{BNAN1}
\left\{
\begin{aligned}
\Delta_{{\mathbb S}^3}U -\lambda U + U^p&=0,\,\,\,\,&& \text{in $\Omega_{\theta_1,\theta_2}$,}\\
U &= 0&&\text{on $\partial \Omega_{\theta_1,\theta_2}$,}
  \end{aligned}
\right.
\end{equation}
in the case $-\lambda_{1}<\lambda\leq 1$, where $\Delta_{{\mathbb S}^3}$ is the Laplace-Beltrami operator, $\lambda_{1}$ is the first eigenvalue of $-\Delta_{{\mathbb S}^3}$ and $\Omega_{\theta_1,\theta_2}$ is an annular domain on ${\mathbb S}^3$: whose great circle distance (geodesic distance) from the North Pole $(0,0,0,1)$ is greater than $\theta_1$ and less than $\theta_2$. 
A solution is said to be ``radial'' if it depends only on this geodesic distance. 

The studies of nonlinear elliptic equations on Riemannian manifolds are recently attracting wide
interest because of the characteristic nature of solutions affected by their curvatures.
The Brezis-Nirenberg problem on constant curvature spaces continues to be studied as a standard problem. 
For the case of hyperbolic space, the existence and non-existence of positive solutions of the Brezis-Nirenberg problems on an entire space or a domain are considered first by Mancini and Sandeep \cite{MR2483639}; see also \cite{MR3016513,MR1937548,MR3956507,SWH2015}. We note that their asymptotic behavior is considered in Bandle and Kabeya \cite{MR3033174}. Further, Morabito \cite{MR3520232} studies a problem of expanding annuli in some geometric spaces, which
include an $n$-dimensional hyperbolic space as a typical case, and shows the existence of non-radial solutions with a bifurcation argument.
We note, Hasegawa \cite{Hasegawa2015,Hasegawa2017} studies the critical exponent of H\'enon type equation for the stability of solutions. 
He also studies the critical exponent with respect to the existence of sign-changing radial solutions of H\'enon type equation \cite{Hasegawa2019}.

On the other hand, for the problems on the sphere, Bandle and Peletier \cite{MR1666821} and Bandle and Benguria \cite{MR1878530} study the scalar-field equation on geodesic balls (spherical caps) in ${\mathbb S}^3$ and Bandle and Wei \cite{BW2008}, Bandle, Kabeya and Ninomiya \cite{BKN2010,BKN2019} study the same problem on spherical caps in $\mathbb{S}^N$; see also Kosaka and Miyamoto \cite{KM2019} and Hirose \cite{Hirose2022}.
For the Brezis-Nirenberg problem on spherical cap, Brezis and Peletier \cite{BP2006} studies the problem in the case of ${\mathbb S}^3$ with critical Sobolev exponent.

In contrast, very little is known about the problems on annular domains in $\mathbb{S}^N$, except the results of Shioji and Watanabe \cite{MR3470747, SW2020-1}.
From this situation, in the present paper, we investigate the multiple existence of positive radial solutions of the Brezis-Nirenberg problem in annular domains, which are symmetric with respect to the equator of ${\Bbb S}^3$ (hence $\theta_{1}=\varepsilon$ and $\theta_{2}=\pi-\varepsilon$).
Since the domains are symmetric with respect to the equator, we can naturally expect the existence of solutions symmetric with respect to the equator, which we call these solutions ``radial even function solutions'' for short. 
As stated, we are concerned with the multiple existence of solutions, though, as well as, we have interests in the qualitative nature of
solutions, concretely the existence and non-existence of radial non-even function solutions.

Main result is as follows:
\begin{thm}\label{thmmain}
For each $p>1$, there exists $\varepsilon_p\in(0,\pi/2)$ such that if $0<\varepsilon\le \varepsilon_{p}$, then 
the following \textup{(i)--(vii)} hold:
\begin{enumerate}
\item[(i)] If $-\lambda_{1}<\lambda\leq 0$ and $p>1$, then problem \eqref{BNAN1}
has a unique positive radial solution and it is an even function, where 
\begin{align}\label{feigen}
\lambda_{1}=\frac{\pi^2}{4\left(\frac{\pi}{2}-\epsilon\right)^{2}}-1>0
\end{align}
is the first eigenvalue of the Laplacian $-\Delta_{{\mathbb S}^3}$; 
\item[(ii)] if $0<\lambda\leq 3/4$ and $1<p\leq 5$, then problem \eqref{BNAN1} has a unique positive radial solution and it is an even function;
\item[(iii)] if $0<\lambda\leq 3/4$ and $p>5$, then problem \eqref{BNAN1} has at least five positive radial solutions (at least three even function solutions and at least two non-even function solutions);
\item[(iv)] if $3/4<\lambda\leq 1$ and $1<p\leq I(\lambda)$, then problem \eqref{BNAN1} has a unique positive radial solution and it is an even function, where
\begin{align}\label{I} 
I(\lambda)=\left(7-6\lambda-\cos\left(\pi\sqrt{1-\lambda}\right)\right)/\left(2\lambda-1-\cos\left(\pi\sqrt{1-\lambda}\right)\right) ; 
\end{align}  
\item[(v)] if $3/4<\lambda\leq 1$ and $3/\lambda+1<p\leq 5$, then problem \eqref{BNAN1} has at least three positive radial solutions (unique even function solution and at least two non-even function solutions);
\item[(vi)] if $3/4<\lambda\leq 1$ and $p>5$ then problem \eqref{BNAN1} has at least five positive radial solutions (at least three even function solutions and at least two non-even function solutions);
\item[(vii)] if $3/4<\lambda\leq 1$ and $1<p\leq 5$, then problem \eqref{BNAN1} has a unique positive radial even function solution.
\end{enumerate}
\end{thm}
Here, we state some remarks about Theorem \ref{thmmain}.
\begin{rem}\label{remth}\quad\\
\vspace{-5mm}
 \begin{enumerate}
\item The exponent $p+1=6$, which separates unique existence and multiple existence of positive even function solutions,
coincides with the critical exponent of Sobolev embedding $H^{1}_{0}(\Omega)\hookrightarrow L^{p+1}(\Omega)$ where $\Omega$ is a three-dimensional bounded domain.    
\item In the cases of Theorem \ref{thmmain} (iv) and (v), the uniqueness of positive radial solutions breaks at some point in the interval $[I(\lambda),3/\lambda +1]$. The exact value of $p$, that separates the ``unique existence'' from ``multiple existence'' of positive radial solutions, remains open. However, if $\lambda$ is sufficiently close to $3/4$, the difference $(3/\lambda +1)-I(\lambda)$ is arbitrarily small; see Figure \ref{fig:sol}.
\item For the case $0<p<1$ (sublinear case), from general results of sublinear ODE (for example, p. 266 of Walter \cite{Walter}), we know that \eqref{BNAN1} has a unique positive radial solution. 
Moreover, from the non-increasing property of $h(t)$ in \eqref{ht} and the application of celebrated Gidas-Ni-Nirenberg theorem \cite{MR544879}, it is an even function.
\end{enumerate}
\end{rem}
\begin{figure}[htb]
\begin{center}
  \includegraphics[scale=0.85 ]{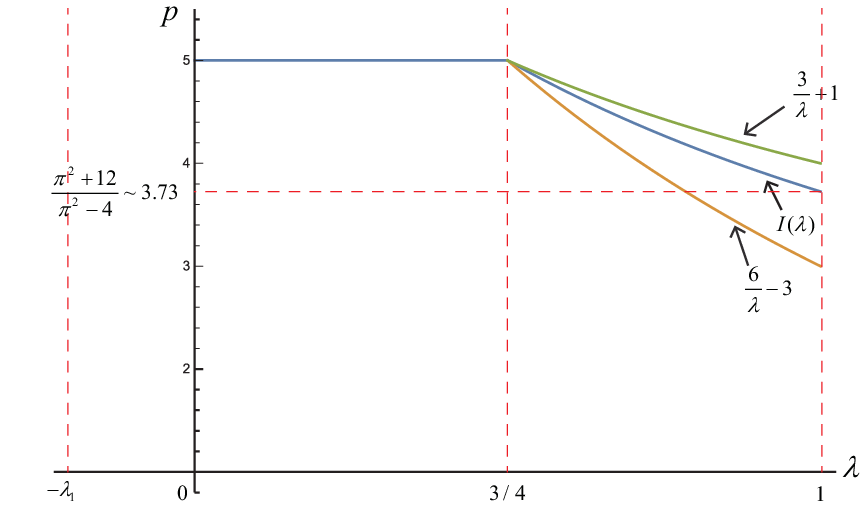}
\caption{Each graphs appearing in Theorem 1.}
\label{fig:sol}
\end{center}
\end{figure}

Here, we note that from Lemma \ref{uniqeven} below, it is demonstrated that 
the solution which satisfy Theorem 1.6 of Brezis-Pletier \cite{BP2006} is never achieved
within the range $-\lambda_1<\lambda\leq 1$:
\begin{thA}[{\cite[Theorem 1.6]{BP2006}}]
Let $m\geq 1$, $p=5$ and $\lambda>m(m+1)$. For every $k\in \{1,2,\ldots,m\}$, there exists a solution $\tilde{u}_{k}$ 
of \eqref{eqr} with boundary condition $\tilde{u}_{k}(-\pi/2)=\tilde{u}_{k}(\pi/2)\geq 0$  and following properties
\begin{enumerate}
\item[(a)]$\tilde{u}_{k}(r)$ has exactly $k$ local maxima (or ``spikes'') on $(-\pi/2,\pi/2)$
\item[(b)]$\tilde{u}_{k}(-r)=\tilde{u}_{k}(r)$,\,\,$(0<r<\pi/2)$
\item[(c)]$\tilde{u}_{k}(\pi/2)<\vert \lambda\vert^{\frac{1}{p-1}}$.
\end{enumerate}
\end{thA}
Concretely, as a corollary of Lemma \ref{uniqeven} we obtain:
\begin{corollary}\label{corop}
Assume $-\lambda_1<\lambda\leq 1$, $1<p\leq 5$. Then, the non-negative even function solution of \eqref{eqr} 
is only a constant solution $u\equiv \vert \lambda\vert^{\frac{1}{p-1}}$ (thus, no ``spike'' solution exists).
\end{corollary}

Before we get into the details, let us sketch the outline of this paper.

First, employing the Poho\v{z}aev type identity, we prove (i), (ii), (iv) and (vii) of Theorem \ref{thmmain}
(We note assertion (ii) follows from Theorem 13 (i)-(b) of \cite{MR3470747}).
We note that there exists a solution that attains the infimum of a Rayleigh quotient.
We call it an even least energy solution. 
When $0<\lambda\leq 1$ and $p>5$, the even least energy solution converges to $0$ 
as $\varepsilon\to0$.
Combining this property with a shooting method, we give a proof of 
(iii) and (vi) of Theorem \ref{thmmain}. 
One of three solutions obtained in (iii) and (vi) of Theorem \ref{thmmain} is 
close to $0$ and other two are close to $\lambda^{1/(p-1)}$.
We note here that $U\equiv \lambda^{1/(p-1)}$ is a constant solution of 
$\Delta_{{\mathbb S}^3}U -\lambda U + U^p=0$.
For the case (v), we show that the Morse index 
of the even least energy solution is greater than $1$ 
in the radial function space.
On the other hand, it is known that the Morse index of 
the (not limited to even) least energy solution is $1$,
and hence the least energy solution must be non-even. 
While the case $p>5$, (vi) of Theorem \ref{thmmain} is obtained 
through evaluating the second variation of the Rayleigh quotient and
using a brow-up-like argument.
\section{Uniqueness results}
In this section, we show the uniqueness results in Theorem \ref{thmmain} (assertions (i), (ii), (iv) and (vii)).
\subsection{Uniqueness of even function solutions}
Let $\tau$ denote a geodesic distance from the North Pole; see Figure \ref{figI}.
Since $(x_{1},x_{2},x_{3},x_{4})\in \Omega_{\varepsilon,\pi-\varepsilon}\subset {\Bbb R}^4$ can be expressed as
\begin{align*}
(x_{1},x_{2},x_{3},x_{4})=\left(\cos\tau,\sin\tau\cos u,\sin\tau\sin u\cos v,\sin\tau\sin u\sin v\right)
\end{align*}
with $(\tau,u,v)$ satisfying $\tau\in(\varepsilon,\pi-\varepsilon)$, $u\in [0,\pi]$, $v\in [0,2\pi)$, 
Riemannian metric tensors of ${\Bbb S}^3$ induced from 
the canonical inner product of ${\Bbb R}^4$
become $g_{\tau\tau}=1, g_{uu}=(\sin\tau)^2, g_{vv}=(\sin\tau)^2(\sin u)^2$, $g_{\tau u}=g_{\tau v}=g_{uv}=0$. 
Hence every positive radial solution, which depends only on $\tau$ of \eqref{BNAN1} satisfies
\begin{figure}[htb]
\begin{center}
  \includegraphics[scale=0.85 ]{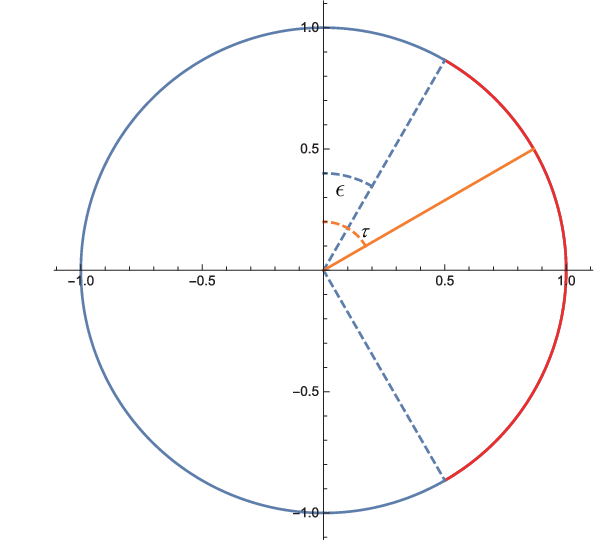}
\caption{The domain $\Omega_{\varepsilon,\pi-\varepsilon}$.}
\label{figI}
\end{center}
\end{figure}
\begin{align}\label{eqtau}
\begin{cases}
u_{\tau\tau}(\tau)+2(\cot \tau)u_{\tau}(\tau)-\lambda u(\tau)+u(\tau)^{p}=0,\,\,\, u(\tau)>0,\,\,\,\,\tau\in \left(\varepsilon,\pi-\varepsilon\right),\\
u\left(\varepsilon\right)=u\left(\pi-\varepsilon\right)=0.
\end{cases}
\end{align}
Set $r=\tau-(\pi/2)$ and $a=(\pi/2)-\varepsilon$, we have from \eqref{eqtau},
\begin{align}\label{eqr}
\begin{cases}
u_{rr}(r)-2(\tan r)u_{r}(r)-\lambda u(r)+u(r)^{p}=0,\,\,\, u(r)>0,\,\,\,\, r\in (-a,a),\\
u(-a)=u(a)=0.
\end{cases}
\end{align} 
Now we define the functions $\psi$ and $\phi$ by
\begin{equation*}
 \psi(r)=\left\{
  \begin{array}{cl}
   \cosh(r\sqrt{\lambda-1}), & \lambda>1, \\[1ex]
   1,& \lambda=1, \\[1ex]
   \cos(r\sqrt{1-\lambda}), & \lambda<1, \\[1ex]
  \end{array}
 \right.
\end{equation*}
and
\begin{equation*}
 \phi(r)=\frac{\psi(r)}{\cos r}.
\end{equation*}
Then $\phi$ is a solution of 
\begin{equation}\label{phidef}
 \phi_{rr} - 2(\tan r)\phi_r- \lambda \phi =0,
\end{equation}
$\phi(r)>0$ for $r\in(-a,a)$ when $\lambda>-\lambda_{1}$ and $\psi$ is a solution of 
\begin{equation*}
 \psi_{rr}(r)+(1-\lambda) \psi(r) = 0,
\end{equation*}
$\psi(r)>0$ for $r\in (-a,a)$ when $\lambda>-\lambda_{1}$.
From \eqref{phidef}, we know that the first eigenvalue $\lambda_{1}$ of $\Delta_{{\Bbb S}^{3}}$ satisfies the relation (hence we obtain \eqref{feigen})
\begin{align*}
a\sqrt{1-\lambda_{1}}=\frac{\pi}{2}.
\end{align*}
Moreover we set
\begin{equation*}
 t(r) = \left\{
  \begin{array}{cl}
   \dfrac{1}{\sqrt{\lambda-1}} \tanh(r\sqrt{\lambda-1}), & \lambda>1, \\[2ex]
   r,& \lambda=1, \\[1ex]
   \dfrac{1}{\sqrt{1-\lambda}} \tan(r\sqrt{1-\lambda}), & \lambda<1, \\[1ex]
  \end{array}
 \right.
\end{equation*}
and let $r(t)$ is the inverse function of $t(r)$.
We note that
\begin{equation*}
 (t(r))_r = \frac{1}{\psi(r)^2} = \frac{1}{(\cos r)^2\phi(r)^2}.
\end{equation*}
Now we set $w(t)=u(r(t))/\phi(r(t))$, that is, $u(r)=w(t(r))\phi(r)$.
Then \eqref{eqr} is rewritten as 
\begin{align}\label{eqw}
\begin{cases}
w_{tt}(t)+(\cos r(t))^{1-p}\psi(r(t))^{p+3}w(t)^{p}=0,\,\,\, w(t)>0,\,\,\,\, r\in (-b,b),\\
w(-b)=w(b)=0,
\end{cases}
\end{align}
where $b:=t(a)$.
In the followings we sometimes use the abbreviation 
\begin{align}\label{ht}
h(t)=(\cos r(t))^{1-p}\psi(r(t))^{p+3}=(\cos r(t))^{1-p}\left(\cos \left(r(t)\sqrt{1-\lambda}\right)\right)^{p+3}
\end{align}
for the sake of simplicity when $\lambda\leq 1$.
We also use the expression
\begin{align}\label{hr}
\bar{h}(r)=(\cos r)^{1-p}\left(\cos \left(r\sqrt{1-\lambda}\right)\right)^{p+3}.
\end{align}
We note that the differential equation of \eqref{eqw} has 
an exact solution.
\begin{lem}\label{lemcosr}
 The function 
 $\lambda^{1/(p-1)}/\phi(r(t))=\lambda^{1/(p-1)}\cos r(t)/\psi(r(t))$ 
 is a solution of
 \begin{equation*}
  w_{tt}(t)+(\cos r(t))^{1-p}\psi(r(t))^{p+3}w(t)^{p}=0.
 \end{equation*}
\end{lem}
We consider the following initial value problem associated with \eqref{eqw}:
\begin{align}\label{initw}
\begin{cases}
w_{tt}(t)+\left(\cos r(t)\right)^{1-p} \psi(r(t))^{p+3} \left| w(t)\right|^{p-1}w(t)=0,\,\,\,\,t\in (0,b),\\
w(0)=\alpha,\,w_{t}(0)=0,
\end{cases}
\end{align}
where $\alpha>0$. 
We denote by $w(r;\alpha)$ a solution of \eqref{initw}
and denote by $Z(\alpha)\in (0,b]$ the first zero in $(0,b]$ of \eqref{initw}, if it exists. 
Since $h\in C^{\infty}[0,b)$, from \cite{Hartman}, the solution $w(t;\alpha)$ of \eqref{initw}
is unique and $w(t;\alpha)$ is $C^{1}$ functions on the set $[0,b)\times (0,\infty)$. 
Hence $Z$ is a continuous function of $\alpha$, if $\alpha$ satisfies $Z(\alpha)<b$.

Since Theorem 1 assumes 
\begin{align}\label{lambdaas}
-\lambda_{1}<\lambda\leq 1, 
\end{align}
from now on we assume \eqref{lambdaas}. 

Now we obtain the following limit property of $Z(\alpha)$ for large $\alpha$:
\begin{lem}\label{zzero}
For all sufficiently large $\alpha>0$, there exists $Z(\alpha)$ and it satisfies
\begin{align*}
\lim_{\alpha\rightarrow\infty}Z(\alpha)=0.
\end{align*}
\end{lem}
\begin{proof}
Although this lemma can be proved almost the same as Lemmas 2.4 and 2.5 of \cite{Tanaka2013}, we give its proof in Appendix A for the sake of convenience. 
\end{proof}
Here, we prove that 
problem \eqref{eqw} always has at least one positive solution $w\in C^{2}[-a,a]$.
\begin{lem}\label{exmin}
Let $p>1$
and $H^{1}_{0}(-b,b)$ be a Sobolev space whose elements vanish at $t=\pm b$, $(0<b<t(\pi/2))$. 
Then, a positive solution $w\in C^{2}[-b,b]$ of \eqref{eqw} always exists and it is a minimizer of the Rayleigh quotient
\begin{align}\label{Reyleigh-ne}
\inf_{\{w\in H^{1}_{0}(-b,b)\,\vert\,w\not\equiv 0\}}R(w):=
\inf_{\{w\in H^{1}_{0}(-b,b)\,\vert\,w\not\equiv 0\}}\frac{\int_{-b}^{b}w_{t}(t)^2dt}{\left(\int_{-b}^{b}
\left(\cos r(t)\right)^{1-p}\psi(r(t))^{p+3}\vert w(t)\vert ^{p+1}dt\right)^{\frac{2}{p+1}}}.
\end{align} 
Also, there exists a positive even function solution $w\in C^{2}[-b,b]$ of \eqref{eqw} which is a minimizer of
\begin{align}\label{Reyleigh}
E(b):=\inf_{\{w\in H^{1}_{0}(-b,b)\,\vert\,w\not\equiv 0, w:\text{even}\}}R(w).
\end{align} 
\end{lem}

We call $w\in C^{2}[-b,b]$ a {\it least energy solution} of \eqref{eqw} if
$w$ is a solution of \eqref{eqw} and a minimizer of the Rayleigh quotient 
\eqref{Reyleigh-ne}.
We also call $w\in C^{2}[-b,b]$ an {\it even least energy solution} of 
\eqref{eqw} if $w$ is a solution of \eqref{eqw} and a minimizer of 
\eqref{Reyleigh}.

\begin{proof}[Proof of Lemma \ref{exmin}.]
We show the minimizer of $R$ exists for the case of \eqref{Reyleigh-ne}. Since we consider a minimization problem, we can consider the infimum of $R$ on the set
\begin{align*}
B_{K}=\left\{w\in H^{1}_{0}(-b,b)\,\vert\,\Vert w\Vert_{H^{1}_{0}(-b,b)}\leq K \text{ and }
\int_{-b}^{b}\left(\cos r(t)\right)^{1-p}\psi(r(t))^{p+3}\vert w(t)\vert ^{p+1}dt=1\right\}
\end{align*}
for $K>0$ is large enough. Since the Sobolev embedding $H^{1}_{0}(-b,b)\subset L^{\infty}(-b,b)$ is compact, $B_{K}$ is a weak compact set. 
Hence, from the weak lower-semi-continuity of $\Vert w\Vert_{H^{1}_{0}(-b,b)}^{2}$, a minimizer $\tilde{w}$ of $R$ exists.  
Noting $\Vert \tilde{w}\Vert_{H^{1}_{0}(-b,b)}=\Vert \vert \tilde{w}\vert\Vert_{H^{1}_{0}(-b,b)}$,
we can assume the minimizer is non-negative. 
Hence, there exists a Lagrange multiplier $\mu>0$ such that $\tilde{w}$ satisfies
\begin{align*}
\begin{cases}
-\tilde{w}_{tt}(r)=\mu \left(\cos r(t)\right)^{1-p}\psi(r(t))^{p+3}\tilde{w}(t)^{p},\,\,\,\,r\in (-b,b),\\
\tilde{w}(-b)=\tilde{w}(b)=0.
\end{cases}
\end{align*}
By taking $w=\mu^{1/(p-1)}\tilde{w}$, we see that $w$ is also a minimizer of $R$ and satisfies \eqref{eqw}. 
Assume $w(r_{0})=0$ for some $r_{0}\in (-b,b)$. Then, from the non-negativity of $w$, $w_{r}(r_{0})=0$.
However, from the uniqueness of initial value problem, we obtain $w\equiv 0$ on $[-b,b]$. This contradicts to
\begin{multline*}
 \int_{-b}^{b}\left(\cos r(t)\right)^{1-p}\psi(r(t))^{p+3}\vert w(t)\vert ^{p+1}dt \\ 
=\mu^{\frac{p+1}{p-1}}\int_{-b}^{b}\left(\cos r(t)\right)^{1-p}\psi(r(t))^{p+3}\vert \tilde{w}(t)\vert^{p+1}dt=\mu^{\frac{p+1}{p-1}}>0.
\end{multline*}
Hence $w$ is positive in $(-b,b)$. The assertion for \eqref{Reyleigh} is shown in a similar way.
\end{proof}
First, we show the assertion (vii) of Theorem \ref{thmmain}.
\begin{lem}\label{uniqeven}
Assume $3/4<\lambda\leq 1$ and $1<p\leq 5$. Then, for $0<b<t(\pi/2)$, an even function solution of \eqref{eqw} is unique.
While in the case $b=t(\pi/2)$, an even function solution of \eqref{eqw} satisfying $w(0)\geq \lambda^{\frac{1}{p-1}}$ is unique and it is $w(t)=\lambda^{\frac{1}{p-1}}(\cos r(t))/\cos(\sqrt{1-\lambda}r(t))$.
\end{lem}
\begin{proof}
We note that if $0<b<(\pi/2)$, an even function solution of \eqref{eqw} always exists from Lemma \ref{exmin}. To prove the uniqueness of even function solutions, we use a Poho\v{z}aev type function of \cite{SW2020-1,SW2020-2}. 
Assume $0<b<t(\pi/2)$ and there exist two distinct solutions $w(t;\alpha_{1})$ and $w(t;\alpha_{2})$ of \eqref{initw}
satisfying $Z(\alpha_{1})=Z(\alpha_{2})=b$, where $0<\alpha_{1}<\alpha_{2}$. 
Since by Lemma \ref{zzero}, $Z(\alpha)\rightarrow 0$ as $\alpha\rightarrow\infty$, we can assume 
(i) $w(t;\alpha_{1})$ and $w(t;\alpha_{2})$ intersect exactly once on $(0,b)$ or (ii) $w(t;\alpha_{1})$ and $w(t;\alpha_{2})$ does not intersect on $(0,b)$ (see Proposition 1.1 and Lemma 1.2 of Kabeya and Tanaka \cite{Kabeya-Tanaka1999}, see also Lemma 1 of Shioji and Watanabe \cite{SW2020-1}). 
We show that (ii) cannot happen. 
To the contrary, we assume (ii). Thus, $w(t;\alpha_{2})>w(t;\alpha_{1})>0$ for $t\in [0,b)$.
By integrating the following identity 
\begin{multline*}
\left[w_{t}(t;\alpha_{2})w(t;\alpha_{1})-w_{t}(t;\alpha_{1})w(t;\alpha_{2})\right]_{t}\\
=h(t)w(t;\alpha_{2})w(t;\alpha_{1})\left(-w(t;\alpha_{2})^{p-1}+w(t;\alpha_{1})^{p-1}\right)
\end{multline*}
on $(0,b)$, we see that the left-hand-side is zero, while the right-hand-side is negative, which is a contradiction.

Next, we consider the case $b=t(\pi/2)$. Assume that there exists $\alpha^{*}>\lambda^{\frac{1}{p-1}}$ such that $Z(\alpha^*)=t(\pi/2)$. Then, again the solution
$w(t;\lambda^{\frac{1}{p-1}})=\lambda^{\frac{1}{p-1}}(\cos r(t))/\cos(\sqrt{1-\lambda}r(t))$ and $w(t;\alpha^{*})$ satisfies the following: (i) intersect once on $(0,t(\pi/2))$ or (ii) they do not intersect each other on $(0,t(\pi/2))$. 
We claim that (ii) does not occur. To the contrary, we assume (ii).
Integrating the following identity  
\begin{multline*}
\left[w_{t}(t;\alpha^{*})w(t;\lambda^{\frac{1}{p-1}})-w_{t}(t;\lambda^{\frac{1}{p-1}})w(t;\alpha^{*})\right]_{t}\\
=h(t)w(t;\alpha^{*})w(t;\lambda^{\frac{1}{p-1}})\left(-w(t;\alpha^{*})^{p-1}+w(t;\lambda^{\frac{1}{p-1}})^{p-1}\right)
\end{multline*}
on $(0,t)$, we obtain
\begin{multline}\label{wint2}
w_{t}(t;\alpha^{*})w(t;\lambda^{\frac{1}{p-1}})-w_{t}(t;\lambda^{\frac{1}{p-1}})w(t;\alpha^{*}) \\
=\int_{0}^{t} h(s)w(s;\alpha^{*})w(s;\lambda^{\frac{1}{p-1}})\left(-w(s;\alpha^{*})^{p-1}+w(s;\lambda^{\frac{1}{p-1}})^{p-1}\right)ds.
\end{multline}
We find that the right-hand-side of \eqref{wint2} is negative. 
We verify that the left-hand-side of \eqref{wint2} tends to zero as $t\rightarrow t(\pi/2)$ (this is unclear, since it might be that $w_{t}(t;\alpha^{*})\rightarrow\infty$ as $t\rightarrow t(\pi/2)$).
We see the second term of the left-hand-side satisfies 
\begin{align*}
 w_{t}(t;\lambda^{\frac{1}{p-1}})w(t;\alpha^{*})\rightarrow 0 \text{ as } t \rightarrow t(\pi/2).
\end{align*}
While, for the first term of the left-hand-side of \eqref{wint2}, noting the concavity of $w$, we obtain
\begin{align*}
&\left|w_{t}\left(t\left(\frac{\pi}{2}-\varepsilon\right);\alpha^{*}\right)w\left(t\left(\frac{\pi}{2}-\varepsilon\right);\lambda^{\frac{1}{p-1}}\right)\right|\\
\leq& \left|\frac{w\left(t(\frac{\pi}{2}-\varepsilon);\alpha^{*}\right)-w\left(t(\frac{\pi}{2});\alpha^{*}\right)}{
t(\pi/2-\varepsilon)-t(\pi/2)}\right|w\left( t\left(\frac{\pi}{2}-\varepsilon\right);\lambda^{\frac{1}{p-1}}\right)\\
=&\left|w\left(t\left(\frac{\pi}{2}-\varepsilon\right);\alpha^{*}\right)\right|\cdot
\left|\frac{-\varepsilon}{t(\pi/2-\varepsilon)-t(\pi/2)}\right|\cdot
\left|\frac{\lambda^{\frac{1}{p-1}}\cos\left(\frac{\pi}{2}-\varepsilon\right)}{\varepsilon \left(\cos\sqrt{1-\lambda}\left(\frac{\pi}{2}-\varepsilon\right)\right)}\right|
\rightarrow 0 \text{ as } \varepsilon\rightarrow 0.
\end{align*}
Hence, the left-hand-side of \eqref{wint2} tends to zero as $t\rightarrow t(\pi/2)$. Therefore, we can assume (i) in both cases $0<b<t(\pi/2)$ and $b=t(\pi/2)$.
From this, we additionally obtain
\begin{align*}
&w_{t}\left(t\left(\frac{\pi}{2}\right);\lambda^{\frac{1}{p-1}}\right)\\
=&\left. \frac{d}{dr}\left(\frac{\lambda^{\frac{1}{p-1}}\cos r}{\cos\left(\sqrt{1-\lambda}r\right)}\right)\right|_{r=\frac{\pi}{2}}\cdot
\left.\frac{dr}{dt}\right|_{t=t(\frac{\pi}{2})}
=-\lambda^{\frac{1}{p-1}}
\cos\left(\sqrt{1-\lambda}\frac{\pi}{2}\right)
\leq w_{t}\left(t\left(\frac{\pi}{2}\right);\alpha^{*}\right)\leq 0.
\end{align*}
Thus, especially we have
\begin{align}\label{wrbdd}
-\infty < w_{t}\left(t\left(\frac{\pi}{2}\right);\alpha^{*}\right)\leq 0. 
\end{align}
Put $w_{1}(t)=w(t;\alpha_{1}), w_{2}(t)=w(t;\alpha_{2})$ in the case $0<b<t(\pi/2)$ and $w_{1}(r)=w(t;\lambda^{\frac{1}{p-1}}), w_{2}(t)=w(t;\alpha^{*})$ in the case $b=t(\pi/2)$. We claim that $w_{1}(t)/w_{2}(t)$ is monotone increasing. By direct computation, it holds
\begin{align}\label{w1-w2}
\left(\frac{w_{1}(t)}{w_{2}(t)}\right)_{t}=\frac{w_{1;t}(r)w_{2}(t)-w_{1}(t)w_{2;t}(t)}{w_{2}(t)^2},
\end{align} 
and
\begin{align*}
\left(w_{1;t}(r)w_{2}(t)-w_{1}(t)w_{2;t}(t)\right)_{t}=h(t)w_{1}(t)w_{2}(t)\left(-w_{1}(t)^{p-1}+w_{2}(t)^{p-1}\right).
\end{align*}
Since $w_{1}$ and $w_{2}$ intersect only once on $(0,b)$, $\left(w_{1;t}(t)w_{2}(t)-w_{1}(t)w_{2;r}(t)\right)_{t}$ 
changes its sign form $+$ to $-$ once on $(0,b)$. Noting \eqref{wrbdd} in the case $b=t(\pi/2)$, we have
\begin{align*}
w_{1;t}(b)w_{2}(b)-w_{1}(b)w_{2;t}(b)=0,\,\,\,\,w_{1;t}(0)w_{2}(0)-w_{1}(0)w_{2;t}(0)=0.
\end{align*}
Hence from \eqref{w1-w2}, we obtain 
\begin{align}\label{monotone1}
 \left(\frac{w_{1}(t)}{w_{2}(t)}\right)_{t}>0,\,\,\,t\in (0,b).
\end{align}
 Here, we introduce a Poho\v{z}aev type function $J(r;w)$ which is used in \cite{SW2020-1,SW2020-2}:
\begin{align}\label{J}
J(t;w)=\frac{1}{2}a(t)w_{t}(t)^{2}+b(t)w_{t}(t)w(t)+\frac{1}{2}c(t)w(t)^{2}+\frac{1}{p+1}a(t)h(t)w(t)^{p+1},
\end{align}
where $w$ is a solution of \eqref{eqw}. If $a(t), b(t), c(t)$ are fixed to
\begin{align*}
a(t)=b-t,\,\,\,b(t)=\frac{1}{2},\,\,\,c(t)=0,
\end{align*}
it holds that
\begin{align}\label{Hdiff}
\frac{dJ(t;w)}{dt}&=\left(\frac{1}{2}a_{t}(t)+b(t)\right)w_{t}(t)^2+\left(b_{t}(t)+c(t)\right)w_{t}(t)w(t)\\
&\quad\ +\frac{1}{2}c_{t}(t)w(t)^2+\left\{-b(t)h(t)+\frac{1}{p+1}\left(a(t)h(t)\right)_{t}\right\}w(t)^{p+1} \nonumber\\
&=H(t)w(t)^{p+1},\nonumber
\end{align}
where 
\begin{align}\label{Hrep}
H(t)=-\frac{1}{2}h(t)+\frac{1}{p+1}\left((b-t)h(t)\right)_{t}
=\frac{1}{p+1}\left(-\frac{p+3}{2}h(t)+(b-t)h_{t}(t)\right).
\end{align}
The last term of \eqref{Hrep} can be rewritten as:
\begin{align}\label{Hrepr}
&\frac{1}{p+1}\left(-\frac{p+3}{2}h(t)+(b-t)h_{t}(t)\right)\\
=&\frac{1}{p+1}\left(-\frac{p+3}{2}\bar{h}(r)+(t(a)-t(r))\bar{h}_{r}(r)\cdot\frac{dr}{dt}\right)\nonumber\\
=&\frac{1}{p+1}\left(-\frac{p+3}{2}\bar{h}(r)+(t(a)-t(r))\bar{h}_{r}(r)\cdot\left(\cos r\sqrt{1-\lambda}\right)^{2}\right)\nonumber
\end{align}
We put 
\begin{align}\label{varr}
\varphi(r,a)= -\frac{p+3}{2}+(t(a)-t(r))\frac{\bar{h}_{r}(r)}{\bar{h}(r)}\left(\cos r\sqrt{1-\lambda}\right)^{2}.
\end{align}
We treat the case $\lambda<1$.
Since by \eqref{hr}, we obtain
\begin{align*}
&\frac{\bar{h}_{r}(r)}{\bar{h}(r)}\left(\cos r\sqrt{1-\lambda}\right)\\
=&-(p+3)\sqrt{1-\lambda} \sin\left(r\sqrt{1-\lambda}\right)
+(p-1)\cos\left(r\sqrt{1-\lambda}\right)\tan(r).
\end{align*}
Also, we have
\begin{align*}
&\left(t(a)-t(r)\right)\left(\cos r\sqrt{1-\lambda}\right)=\left(\frac{\tan\left(a\sqrt{1-\lambda}\right)-\tan\left(r\sqrt{1-\lambda}\right)}{\sqrt{1-\lambda}}\right)\left(\cos (r\sqrt{1-\lambda})\right)\\
=&\frac{\sec(a\sqrt{1-\lambda})\sin\left((a-r)\sqrt{1-\lambda}\right)}{\sqrt{1-\lambda}}.
\end{align*}
Substituting these to \eqref{varr}, we have the following expression:
\begin{align}\label{phie}
\varphi(r,a)=& -\frac{p+3}{2}-\frac{1}{\sqrt{1-\lambda}}\cdot\sec(a\sqrt{1-\lambda})\sin\left((a-r)\sqrt{1-\lambda}\right)\cdot\\
&\quad\quad\cdot \left[(p+3)\sqrt{1-\lambda} \sin\left(r\sqrt{1-\lambda}\right) 
-(p-1)\cos\left(r\sqrt{1-\lambda}\right)\tan(r)\right] .\nonumber
\end{align}
We show that
\begin{align}\label{Hneg}
\varphi(r,a)<0,\quad r\in(0,a),
\end{align}
and hence $H(t)<0$ on $(0,b)$. We temporally assume that \eqref{Hneg} holds. 
Define $d=w_{1}(0)/w_{2}(0)<1$, then from \eqref{monotone1}, \eqref{Hdiff}, \eqref{Hneg} and \eqref{wrbdd} (this inequality is used to show $J(b;w_{2})=0$, in the case $b=t(\pi/2)$), we obtain
\begin{align*}
0&<\int_{0}^{b}H(r)\left(d^{p+1}-\left(\frac{w_{1}(t)}{w_{2}(t)}\right)^{p+1}\right)w_{2}(t)^{p+1}dt \\
&=\int_{0}^{b}\left(d^{p+1}\frac{dJ(t;w_{2})}{dr}-\frac{dJ(t;w_{1})}{dt}\right)dt\\
&=d^{p+1}\left(J(b;w_{2})-J(0;w_{2})\right)-J(b;w_{1})+J(0;w_{1})\\
&=-d^{p+1}J(0;w_{2})+J(0;w_{1})\\
&=-\left(\frac{w_{1}(0)}{w_{2}(0)}\right)^{p+1}\frac{bh(0)}{p+1}w_{2}(0)^{p+1}+\frac{bh(0)}{p+1}w_{1}(0)^{p+1}=0.
\end{align*}
Hence, we obtain a contradiction. 

Now, we show \eqref{Hneg}.
It is easy to see that for $r\in S_{+}\cap (0,a)$, it holds that $\varphi(r,a)<0$, where
\begin{align*}
S_{+}=\{r\in (0,\pi/2)\,\vert\, (p+3)\sqrt{1-\lambda} \sin\left(r\sqrt{1-\lambda}\right) 
-(p-1)\cos\left(r\sqrt{1-\lambda}\right)\tan(r)\geq 0\} .
\end{align*} 
For fixed $r\in (0,a)\setminus S_{+}$, $\varphi(r,a)$ is monotone increasing with respect to $a$, since
\begin{align*}
\frac{\partial\varphi}{\partial a}(r,a)=&-\left(\sec\left(a\sqrt{1-\lambda}\right)\right)^{2}
\left(\cos\left(r\sqrt{1-\lambda}\right)\right)\cdot\\
&\quad\quad\cdot\left[(p+3)\sqrt{1-\lambda} \sin\left(r\sqrt{1-\lambda}\right) 
-(p-1)\cos\left(r\sqrt{1-\lambda}\right)\tan(r)\right]>0.
\end{align*}
Thus it is enough to show that $\varphi(r,\pi/2)<0$ for $r\in (0,\pi/2)\setminus S_{+}$ (indeed, we show in the followings that $\varphi(r,\pi/2)<0$ for $r\in (0,\pi/2)$).
Since it holds $\sin\left(\left(\pi/2-r\right)\sqrt{1-\lambda}\right)<\sqrt{1-\lambda}\cot(r)$ for $r\in (0,\pi/2)$, 
\begin{align}\label{eq0728}
\sin\left(\left(\frac{\pi}{2}-r\right)\sqrt{1-\lambda}\right)\tan(r)<\sqrt{1-\lambda},\quad\quad r\in \left(0,\frac{\pi}{2}\right). 
\end{align}
Applying \eqref{eq0728}, we obtain 
\begin{align*}
&\quad \varphi\left(r,\frac{\pi}{2}\right)< -\frac{p+3}{2}-\frac{\sec\left(\frac{\pi}{2}\sqrt{1-\lambda}\right)}{\sqrt{1-\lambda}}\cdot\\
&\cdot \left[(p+3)\sqrt{1-\lambda}\sin\left(r\sqrt{1-\lambda}\right) \sin\left(\left(\frac{\pi}{2}-r\right)\sqrt{1-\lambda}\right)-(p-1)\sqrt{1-\lambda}\cos\left(r\sqrt{1-\lambda}\right)\right]\\
&=\frac{1}{2}\sec\left(\frac{\pi}{2}\sqrt{1-\lambda}\right)
\left[-(p+3)\cos\left(\left(\frac{\pi}{2}-2r\right)\sqrt{1-\lambda}\right)
+2(p-1)\cos\left(r\sqrt{1-\lambda}\right)
\right]\\
&=:\phi(r).
\end{align*}
By direct computation, it holds
\begin{align*}
\phi_{r}(r)=&-\sqrt{1-\lambda}\sec\left(\frac{\pi}{2}\sqrt{1-\lambda}\right)
\left[(p+3)\sin\left(\left(\frac{\pi}{2}-2r\right)\sqrt{1-\lambda}\right) 
+(p-1)\sin\left(r\sqrt{1-\lambda}\right)
\right]\\
\phi_{rr}(r)=&
(1-\lambda)\sec\left(\frac{\pi}{2}\sqrt{1-\lambda}\right)
\left[2(p+3)\cos\left(\left(\frac{\pi}{2}-2r\right)\sqrt{1-\lambda}\right) 
-(p-1)\cos\left(r\sqrt{1-\lambda}\right)\right].
\end{align*}
We show $\phi(r)<0$ for $r\in [\pi/4,\pi/2)$. Since 
\begin{align*}
0\leq \left(2r-\frac{\pi}{2}\right)\sqrt{1-\lambda}<r\sqrt{1-\lambda}<\frac{\pi}{2},\quad r\in [\pi/4,\pi/2),
\end{align*}
$\phi_{rr}(r)>0$ on $[\pi/4,\pi/2)$. Moreover, we have
\begin{align*}
\phi_{r}\left(\frac{\pi}{4}\right)=&-\sqrt{1-\lambda}(p-1)\sec\left(\frac{\pi}{2}\sqrt{1-\lambda}\right)\sin\left(\frac{\pi}{2}\sqrt{1-\lambda}\right)<0,\\
\phi_{r}\left(\frac{\pi}{2}\right)=&4\sqrt{1-\lambda}\tan\left(\frac{\pi}{2}\sqrt{1-\lambda}\right)>0.
\end{align*}
So, $\phi$ has a unique minimum point on $(\pi/4,\pi/2)$. On the other hand, it holds
\begin{align*}
\phi\left(\frac{\pi}{4}\right)=&\frac{1}{2}\sec\left(\frac{\pi}{2}\sqrt{1-\lambda}\right)
\left[-3-p+2(p-1)\cos\left(\frac{\pi}{4}\sqrt{1-\lambda}\right)\right]\\
<&\frac{1}{2}\sec\left(\frac{\pi}{2}\sqrt{1-\lambda}\right)
\left[-3-p+2(p-1)\right]\\
=&\frac{1}{2}\sec\left(\frac{\pi}{2}\sqrt{1-\lambda}\right)(p-5)\leq 0,\\
\phi\left(\frac{\pi}{2}\right)=&\frac{1}{2}(p-5)\leq 0.
\end{align*}
Thus $\varphi(r,\pi/2)<\phi(r)\leq 0$ on $[\pi/4,\pi/2)$. In the case of $r\in(0,\pi/4)$, it holds
\begin{align*}
\sin\left(r\sqrt{1-\lambda}\right)>\frac{4}{\pi}\sin\left(\frac{\pi}{4}\sqrt{1-\lambda}\right)r \text{ and }
\tan(r)<\frac{4}{\pi}r,
\end{align*}
so for $r\in(0,\pi/4)$, we obtain
\begin{align*}
\varphi(r,\pi/2)<&-\frac{1}{2}(p+3)-\frac{1}{\sqrt{1-\lambda}}
\sec\left(\frac{\pi}{2}\sqrt{1-\lambda}\right)\sin\left(\left(\frac{\pi}{2}-r\right)\sqrt{1-\lambda}\right)\cdot\\
&\cdot
\left[
\sqrt{1-\lambda}(p+3)\left(\frac{4}{\pi}\sin\left(\frac{\pi}{4}\sqrt{1-\lambda}\right)r\right)-(p-1)\left(\frac{4}{\pi}r\right)
\right]\\
=&-\frac{1}{2}(p+3)-\frac{1}{\sqrt{1-\lambda}}
\sec\left(\frac{\pi}{2}\sqrt{1-\lambda}\right)\
\left[r\sin\left(\left(\frac{\pi}{2}-r\right)\sqrt{1-\lambda}\right)\right]\cdot\\
&\cdot
\frac{4}{\pi}\left[
\sqrt{1-\lambda}(p+3)\sin\left(\frac{\pi}{4}\sqrt{1-\lambda}\right)-(p-1)
\right]=:\tilde{\phi}(r)
\end{align*}
Since 
$\left[r\sin\left(\left(\frac{\pi}{2}-r\right)\sqrt{1-\lambda}\right)\right]$ is monotone
increasing for $r\in (0,\pi/4)$, it is enough to show $\tilde{\phi}(0)\leq 0$ and $\tilde{\phi}(\pi/4)<0$.
We have 
\begin{align*}
&\tilde{\phi}\left(\frac{\pi}{4}\right)=-\frac{p+3}{2}\\
-&\frac{1}{\sqrt{1-\lambda}}
\sec\left(\frac{\pi}{2}\sqrt{1-\lambda}\right)\sin\left(\frac{\pi}{4}\sqrt{1-\lambda}\right)
\left[
\sqrt{1-\lambda}(p+3)\sin\left(\frac{\pi}{4}\sqrt{1-\lambda}\right)-(p-1)
\right]
,\\
&\tilde{\phi}(0)=-\frac{p+3}{2}<0.
\end{align*}
We can observe that $\tilde{\phi}(\pi/4)$ is a first order polynomial with respect to $p$ and its coefficient is 
\begin{align*}
\frac{\sec\left(\frac{\pi}{2}\sqrt{1-\lambda}\right)}{2\sqrt{1-\lambda}}\left(-\sqrt{1-\lambda}+2\sin\left(
\frac{\pi}{4}\sqrt{1-\lambda}
\right)\right)>0.
\end{align*}
Hence we show that $\tilde{\phi}(\pi/4)<0$ for $p=5$. Substituting $p=5$ to $\tilde{\phi}(\pi/4)$ we obtain
\begin{align*}
\frac{4\sec\left(\frac{\pi}{2}\sqrt{1-\lambda}\right)}{\sqrt{1-\lambda}}\left[
\sin\left(\frac{\pi}{4}\sqrt{1-\lambda}\right)-\sqrt{1-\lambda}
\right]<0.
\end{align*}
For the case $\lambda=1$, it holds
\begin{align}\label{lam1}
\varphi(r,a)=-\frac{p+3}{2}+(p-1)(a-r)(\tan r)\leq -\frac{p+3}{2}+(p-1)(\frac{\pi}{2}-r)(\tan r).
\end{align}
We can easily show the right-hand-side of \eqref{lam1} is negative for $r\in (0,\pi/2)$ and $1<p\leq 5$.
This completes the proof.
\end{proof}
We prove Corollary \ref{corop}.

\noindent
{\it Proof of Corollary \ref{corop}.} First we note the proof of Lemma \ref{uniqeven} is also applicable to the
case $-\lambda_{1}<\lambda\leq 3/4$. Recall that in the case $b=t(\pi/2)$ and $w(0)\geq \vert\lambda\vert^{\frac{1}{p-1}}$,
even function solution of \eqref{eqw} is unique and hence even function solution of \eqref{eqr} with $a=\pi/2$ is unique ($u\equiv \vert\lambda\vert^{\frac{1}{p-1}}$). 
The assumption $w(0)\geq \vert\lambda\vert^{\frac{1}{p-1}}$ was used to derive the relation \eqref{wrbdd}.
Thus, in the case $w(0)< \vert\lambda\vert^{\frac{1}{p-1}}$, if \eqref{wrbdd} holds, uniqueness of even function solution of \eqref{eqw} also follows.
Now suppose $w_{t}(t(\pi/2;w(0)))=-\infty$. Since
\begin{align}
w_{t}(t)=u_{r}(r(t))(\cos r(t))\psi(r(t))-u(r(t))\left[\psi_{r}(r(t))(\cos r(t))+\psi(r(t))(\sin r(t))\right], 
\end{align}
and 
\begin{align*}
 \lim_{t\rightarrow t(\pi/2)}\vert u(r(t))\left[\psi_{r}(r(t))(\cos r(t))+\psi(r(t))(\sin r(t))\right]\vert < \infty,
\end{align*}
it holds
\begin{align*}
 -u_{r}(r)=o\left((\cos r)^{-1}\right)\,\,\text{ as } r\rightarrow \pi/2.
\end{align*}
Though, this implies $\lim_{r\rightarrow \pi/2}u(r)=-\infty$. Hence it can not be a non-negative solution of \eqref{eqr}.
\hfill \qed
\begin{lem}\label{endpositive}
Assume $p>3$ and $0<\lambda\leq 1$. Then, 
the first zero $Z(\alpha)$ of the solution to \eqref{initw} exists in $(0,t(\pi/2)]$.
\end{lem}
\begin{proof}

Assume to the contrary $w(t)>0$ for $t\in [0,t(\pi/2)]$.
For simplicity, we denote $\lim_{t\rightarrow t(\pi/2)}w(t)=w_0>0$.
We note it holds
\begin{equation*}
 \cos(r(t)) \ge \frac{2}{\pi}\left(\frac{\pi}{2}-r(t)\right), \quad t \in [0,t(\pi/2)].
\end{equation*}
and 
\begin{align*}
\psi(r(t))=\cos(r(t)\sqrt{1-\lambda})\geq \cos\left(\frac{\pi}{2}\sqrt{1-\lambda}\right)=m,\quad  t \in [0,t(\pi/2)].
\end{align*}
Then, from \eqref{initw} for $t\in(0,t(\pi/2))$, it holds
\begin{align}\label{wint}
-w(t)+w(0)&=\int_{0}^{t}\int_{0}^{\tau}\left(\cos r(s)\right)^{1-p}\psi(r(s))^{p+3}w(s)^{p}dsd\tau\\
&> \left(\frac{2}{\pi}\right)^{1-p} m^{p+3} w_0^p \int_{0}^{t}\int_{0}^{\tau} \left(\frac{\pi}{2}-r(s)\right)^{1-p} dsd\tau \nonumber \\
&= \left(\frac{2}{\pi}\right)^{1-p} m^{p+3} w_0^p \int_{0}^{t}\int_{0}^{r(\tau)}
\left(\frac{\pi}{2}-r\right)^{1-p}\cdot\frac{1}{\left(\cos r\sqrt{1-\lambda}\right)^{2}}drd\tau\nonumber\\
&> \left(\frac{2}{\pi}\right)^{1-p} m^{p+3} w_0^p \int_{0}^{t}\int_{0}^{r(\tau)}
\left(\frac{\pi}{2}-r\right)^{1-p}drd\tau\nonumber\\
&=\left(\frac{2}{\pi}\right)^{1-p} m^{p+3} w_0^p\cdot\frac{1}{p-2} \int_{0}^{t}
\left(\frac{\pi}{2}-r(\tau)\right)^{2-p}-\left(\frac{\pi}{2}\right)^{2-p}d\tau\nonumber\\
&=\left(\frac{2}{\pi}\right)^{1-p} m^{p+3} w_0^p\cdot\frac{1}{p-2} \int_{0}^{r(t)}
\left[
\left(\frac{\pi}{2}-r\right)^{2-p}-\left(\frac{\pi}{2}\right)^{2-p}
\right]
\cdot\frac{1}{\left(\cos r\sqrt{1-\lambda}\right)^{2}}dr\nonumber\\
&>\left(\frac{2}{\pi}\right)^{1-p} m^{p+3} w_0^p \cdot\frac{1}{p-2} \int_{0}^{r(t)}
\left(\frac{\pi}{2}-r\right)^{2-p}-\left(\frac{\pi}{2}\right)^{2-p} dr.\nonumber
\end{align} 
Letting $t\rightarrow t(\pi/2)$, hence $r(t)\rightarrow \pi/2$, the left-hand-side of \eqref{wint} is finite, though the right-hand-side goes to infinity since $3<p$.
Hence a contradiction.
\end{proof}
\subsection{General uniqueness}
In this section, we consider the uniqueness of positive radial solution of \eqref{eqw}, without restricting to even functions.
First, we show Theorem \ref{thmmain}-(ii).
\begin{proof}[Proof of Theorem \ref{thmmain}-(ii).]
The claim is proved by Theorem 13-(i)-(b) of \cite{MR3470747}, which states that if
$p\leq \min\{6/\lambda-3,5 \}$, then \eqref{eqw} has a unique positive radial solution.
We note if $0<\lambda\leq 3/4$, $\min\{6/\lambda-3,5 \}=5$. Thus, we have proved the assertion.
\end{proof}
Next, we show Theorem \ref{thmmain}-(i).
\begin{proof}[Proof of Theorem \ref{thmmain}-(i).]
To the contrary, assume that \eqref{eqw} has two distinct solutions $w_{1}$ and $w_{2}$ satisfying $w_{1;t}(-b)<w_{2;t}(-b)$. Then
it holds from \cite[Lemma 2]{SW2020-1},
\begin{align}\label{monotone2}
 \frac{d}{dt}\left(\frac{w_{1}(t)}{w_{2}(t)}\right)>0,\,\,\,\,t\in (-b,b).
\end{align}
Putting 
\begin{align*}
a(t)=b^2-t^2,\,\,b(t)=t,\,\,c(t)=-1
\end{align*}
in \eqref{J}, 
we obtain \eqref{Hdiff} as in Lemma \ref{uniqeven}, where 
\begin{align}\label{Hrep2}
H(t)=-b(t)h(t)+\frac{1}{p+1}\left(a(t)h(t)\right)_{t}=-\frac{p+3}{p+1}th(t)+\frac{1}{p+1}\left(b^2-t^2\right)h_{t}(t).
\end{align} 
The last term of \eqref{Hrep2} can be rewritten as:
\begin{align}
&=-\frac{p+3}{p+1}th(t)+\frac{1}{p+1}\left(b^2-t^2\right)h_{t}(t)\\
&=\frac{1}{p+1}\left(-(p+3)t(r)\bar{h}(r)+\left(t(a)^2-t(r)^2\right)\bar{h}_{r}(r)\cdot\frac{dr}{dt}\right)
\nonumber\\
&=\frac{1}{p+1}\left(-(p+3)t(r)\bar{h}(r)+\left(t(a)^2-t(r)^2\right)\bar{h}_{r}(r)\cdot \left(\cos r\sqrt{1-\lambda}\right)^{2}\right).\nonumber
\end{align} 
We note it holds
\begin{align*}
&\bar{h}_{r}(r)\left(\cos r\sqrt{1-\lambda}\right)=\left(
(\cos r)^{1-p}\left(\cos r\sqrt{1-\lambda}\right)^{p+3}\right)_{r}\left(\cos r\sqrt{1-\lambda}\right)\\
&=(\cos r)^{-p}\left(\cos r\sqrt{1-\lambda}\right)^{p+3}\cdot\\
&\cdot\left[
(p-1)\left(\cos r\sqrt{1-\lambda}\right)(\sin r)-(p+3)\sqrt{1-\lambda}(\cos r)\left(\sin r\sqrt{1-\lambda}\right)
\right],\\
&\quad\\
&\left(t(a)^2-t(r)^2\right)\left(\cos r\sqrt{1-\lambda}\right)=
\left[
\left(\frac{\tan a\sqrt{1-\lambda}}{\sqrt{1-\lambda}}\right)^2
-\left(\frac{\tan r\sqrt{1-\lambda}}{\sqrt{1-\lambda}}\right)^2
\right]\left(\cos r\sqrt{1-\lambda}\right)\\
&=\frac{1}{1-\lambda}
\left[\left(\cos r\sqrt{1-\lambda}\right)\left(\sec a\sqrt{1-\lambda}\right)^{2} 
-(\sec r\sqrt{1-\lambda})\right].
\end{align*}
Thus, we have 
\begin{align}\label{Hrep2b}
H(t)=&\frac{(\cos r(t))^{-p}\left(\cos r(t)\sqrt{1-\lambda}\right)^{p+3}}{(p+1)(1-\lambda)}\cdot\\
&\cdot \Bigl[-(p-1)(\sin r(t))+\left(\sec a\sqrt{1-\lambda}\right)^{2}\left(\cos r(t)\sqrt{1-\lambda}\right)\cdot\nonumber\\
&\cdot\Bigl(
(p-1)\left(\cos r(t)\sqrt{1-\lambda}\right)\left(\sin r(t)\right)-(p+3)\sqrt{1-\lambda}\left(\sin r(t)\sqrt{1-\lambda}\right)\left(\cos r(t)\right)
\Bigr) \Bigr].\nonumber
\end{align}
We put
\begin{align}\label{genuniq}
\varphi(r,a):=&
-(p-1)(\sin r)+\left(\sec a\sqrt{1-\lambda}\right)^{2}\left(\cos r\sqrt{1-\lambda}\right)\cdot\\
&\cdot\Bigl(
(p-1)\left(\cos r\sqrt{1-\lambda}\right)\left(\sin r\right)-(p+3)\sqrt{1-\lambda}\left(\sin r\sqrt{1-\lambda}\right)\left(\cos r\right)
\Bigr)\nonumber
\end{align}
From \eqref{Hrep2b}, we observe $H$ is an odd function, so if $H(t)<0$ for $t\in (0,b)$, we see that $H$ changes its sign from $+$ to $-$ only once at $t=0$ on $(-b,b)$.
We temporally assume that $H(t)<0$ for $t\in (0,b)$.
Define $d=w_{1}(0)/w_{2}(0)$, then from \eqref{monotone2} and negativity of $H$ on $(0,b)$ (and positivity of $H$ on $(-b,0)$), we obtain
\begin{align*}
&0<\int_{-b}^{b}H(t)\left(d^{p+1}-\left(\frac{w_{1}(t)}{w_{2}(t)}\right)^{p+1}\right)w_{2}(t)^{p+1}dt\\
&=d^{p+1}\left(J(b;w_{2})-J(-b;w_{2})\right)-J(b;w_{1})+J(-b;w_{1})=0.
\end{align*}
Hence a contradiction. Now, we show $H(t)<0$ for $t\in (0,b)$ (actually, we show $\varphi(r,a)<0$ for $r\in (0,a)$).
We can see that (dropping the term $-2(p-1)(\sin r)$), 
\begin{align*}
\varphi(r,a)<& 
2\left(\sec a\sqrt{1-\lambda}\right)^{2}\left(\cos r\sqrt{1-\lambda}\right)\cdot\\
&\cdot\Bigl(
(p-1)\left(\cos r\sqrt{1-\lambda}\right)\left(\sin r\right)-(p+3)\sqrt{1-\lambda}\left(\sin r\sqrt{1-\lambda}\right)\left(\cos r\right)
\Bigr)
.
\end{align*}
Hence it is enough to show that
\begin{align*}
p-1-(p+3)\sqrt{1-\lambda}(\cot r)(\tan r\sqrt{1-\lambda})<0,\quad r\in (0,a).
\end{align*}
Since 
\begin{align*}
-\sqrt{1-\lambda}(\sin 2r)+\sin 2r\sqrt{1-\lambda}<0,\quad r\in (0,a), 
\end{align*}
it holds
\begin{align*}
\left(-(\cot r)\tan r\sqrt{1-\lambda}\right)_{r}<0, ,\quad r\in (0,a).
\end{align*}
Since 
\begin{align*}
p-1-(p+3)\lim_{r\rightarrow 0}\left(\sqrt{1-\lambda}(\cot r)(\tan r\sqrt{1-\lambda})\right)=
-4+(p+3)\lambda <0,
\end{align*}
we have shown the assertion.
\end{proof}

Assertion (iv) of Theorem \ref{thmmain} is rather comlicated, thus we show this part in the 
Appendix B.
\section{Multiple existence Results}
\subsection{Multiple existence of non-even function solutions}

In this subsection, a part of assertions (iii), (v) and (vi) of Theorem \ref{thmmain} is proved.
First, we introduce the Morse index of a solution to problem \eqref{eqw}.
The Morse index of a solution $w$ to \eqref{eqw} is defined as the number of negative eigenvalues to 
\begin{align}\label{LL}
\begin{cases}
 \Phi_{tt}(t)+p\left(\cos r(t)\right)^{1-p}\psi(r(t))^{p+3}w(t)^{p-1}\Phi(t)+\mu\Phi(t)=0,\quad t\in (-b,b),\\
 \Phi(-b)=\Phi(b)=0.
\end{cases}
\end{align}
The following is known.
\begin{lem}\label{lemMI1}
The Morse index of the least energy solution to the Rayleigh quotient 
\eqref{Reyleigh-ne} is one.
\end{lem}
For the proof of this lemma, see; Lemma 3 in Takahashi \cite{Ta2013}.
Using Lemma \ref{lemMI1}, we obtain the following proposition which shows assertion (v) of Theorem \ref{thmmain}.

By the following proposition, we obtain (iv) of Theorem \ref{thmmain}. 
\begin{proposition}\label{propnoneven3}
Let $3/4<\lambda\leq 1$.
For each $p$ satisfying
\begin{align}
\frac{3}{\lambda}+1<p\leq 5,
\end{align}
there exists $\varepsilon_{p}\in (0,\pi/2)$ such that if $0<\varepsilon<\varepsilon_{p}$ and $a=\pi/2 - \varepsilon$
are satisfied, then problem \eqref{eqw} has at least three positive radial solutions (one solution is a unique even function solution and other two solutions are non-even function solutions which are minimizers of \eqref{Reyleigh-ne}). 
\end{proposition}
\begin{proof}
The uniqueness of positive even function solution for $1<p\leq 5$ follows from Lemma \ref{uniqeven}. 
By Lemmas 4 and 5, for each $\alpha>\lambda^{1/(p-1)}$, the first zero 
$Z(\alpha)$ of the solution of (7) exists in $(0,t(\pi/2))$.
Hence, $Z(\alpha)$ is continuous in $\alpha>\lambda^{1/(p-1)}$.
Since $w(r;\lambda^{1/(p-1)})=\lambda^{1/(p-1)}\cos r(t)/\psi(r(t))>0$ for 
$0 \le t <t(\pi/2)$ by Lemma \ref{lemcosr}, 
we deduce that $Z(\alpha)\to t(\pi/2)$ as $\alpha\to{\lambda^{1/(p-1)}}^+$.
We claim that $Z(\alpha)$ is strictly decreasing in $\alpha>\lambda^{1/(p-1)}$.
Assume that there exist $\alpha_1$ and $\alpha_2$ such that
$\lambda^{1/(p-1)}<\alpha_1<\alpha_2$ and $Z(\alpha_1)\le Z(\alpha_2)$.
Lemma 2 implies that there exists $\alpha_3 \ge \alpha_2$
such that $Z(\alpha_1)=Z(\alpha_3)$, which means that 
$w(r;\alpha_1)$ and $w(r;\alpha_3)$ are two distinct even function solutions of
(5) with $b=Z(\alpha_1)$.
This contradicts Lemma 4.
Hence, $Z(\alpha)$ is strictly decreasing in $\alpha>\lambda^{1/(p-1)}$ 
as claimed.
Since $Z(\alpha)\to\pi/2$ as $\alpha\to{\lambda^{1/(p-1)}}^+$ and 
$Z(\alpha)\to0$ as $\alpha\to\infty$, for each $\varepsilon>0$, 
there exists a unique $\alpha_\varepsilon>\lambda^{1/(p-1)}$ such that 
$Z(\alpha_\varepsilon)=t(\pi/2)-\varepsilon$ 
and $\alpha_\varepsilon \to \lambda^{1/(p-1)}$ as $\varepsilon\to0^+$.

Since $p>(3/\lambda)+1$, there exists $\delta\in(0,1)$ such that
\begin{equation*}
 \sqrt{[p\lambda(1-\delta)-\lambda+1]}\left( \frac{\pi}{2}-\delta \right) > \pi.
\end{equation*}
Set
\begin{equation*}
 v(r) = \sin\left(\sqrt{[p\lambda(1-\delta)-\lambda+1]}r \right).
\end{equation*}
Then $v$ is a solution of
\begin{equation*}
 v_{rr}(r) + [p\lambda(1-\delta)-\lambda+1] v(r) = 0
\end{equation*}
and has three zeros $-z_1$, $0$, $z_1$ in 
$(-(\pi/2)+\delta,(\pi/2)-\delta)$ for some $z_1>0$.
We set $\widetilde{v}(r)=v(r)/\cos r$.
Then $\widetilde{v}$ satisfies
\begin{equation*}
 \widetilde{v}_{rr}(r) - 2(\tan r)\widetilde{v}_r(r) 
 - \lambda\widetilde{v}(r) + p\lambda(1-\delta)\widetilde{v}(r) = 0, \quad 
 r \in (-(\pi/2)+\delta,(\pi/2)-\delta).
\end{equation*}
Moreover, we set $V(t)=\widetilde{v}(r(t))/\phi(r(t))$.
Then $V$ is a solution of
\begin{equation*}
 V_{tt}(t) + p\lambda(1-\delta) \psi(r(t))^4 V(t) = 0, \quad
 t \in (-t((\pi/2)-\delta),t((\pi/2)-\delta))
\end{equation*}
and has three zeros $-t(z_1)$, $0$, $t(z_1)$ in 
$(-t((\pi/2)-\delta),t((\pi/2)-\delta))$.

By the continuous dependence of solutions on initial conditions,
$w(r;\alpha_\varepsilon)$ converges to 
$w(r;\lambda^{1/(p-1)})=\lambda^{1/(p-1)}\cos r(t)/\psi(r(t))$ uniformly on 
$[-t(z_1),t(z_1)]$ as $\varepsilon\to0^+$.
Hence, there exists $\varepsilon_p \in (0,\pi/2)$ such that
if $\varepsilon \in (0,\varepsilon_p)$, then 
$a=(\pi/2)-\varepsilon>z_1$ and
\begin{equation}\label{cosw1}
 (\cos r(t))^{1-p} \psi(r(t))^{p-1} w(t;\alpha_{\varepsilon})^{p-1} 
  \ge \lambda\left(1-\frac{\delta}{2}\right), \quad t \in [-t(z_1),t(z_1)]. 
\end{equation}
Hereafter we suppose that $\varepsilon \in (0,\varepsilon_p)$.
Let $\mu_2$ be the second eigenvalue of \eqref{LL} with 
$w(t)=w(t;\alpha_\varepsilon)$.
We claim that $\mu_2<0$.
Assume that $\mu_2\ge0$.
Then from \eqref{cosw1}
\begin{equation*}
 p(\cos r(t))^{1-p} \psi(r(t))^{p+3} w(t;\alpha_\varepsilon)^{p-1} + \mu_2
 > p\lambda(1-\delta) \psi(r(t))^4, \quad r \in [-t(z_1),t(z_1)]. 
\end{equation*}
The Sturm comparison theorem implies that 
every eigenfunction of \eqref{LL} with $w(t)=w(t;\alpha_\varepsilon)$ 
corresponding to $\mu_2$ has at least two zero in 
$(-t(z_1),t(z_1))\subset (-t(a),t(a))$, which is a contradiction (from Sturm-Liouville theory, the second eigenfunction of \eqref{LL} has exactly one zero in $(-t(a),t(a))$).
Therefore, $\mu_2<0$ as claimed.
Recalling Lemma \ref{lemMI1}, we see that $w(t;\alpha_\varepsilon)$ is not 
a least energy solution of the Rayleigh quotient \eqref{Reyleigh-ne}.
Since $w(t;\alpha_\varepsilon)$ is a unique even function solution of 
\eqref{eqw}, we conclude that every least energy solution is non-even function 
solution of \eqref{eqw}.
Since we know that if $w(r)$ is a solution of \eqref{eqw}, then so is $w(-r)$. 
Thus \eqref{eqw} has at least two non-even function solutions.
\end{proof}
We show the numerical experimental results for the case $\lambda=1$. So, in this case $3/\lambda+1=4$. In each figure, the left-hand-side figure shows the graph of
the solution $w$ of \eqref{eqw}, while the right-hand-side figure shows the graph of the solution $u$ of \eqref{eqr}.
 \begin{figure}[htb]
 \begin{center}
  \includegraphics[scale=0.6]{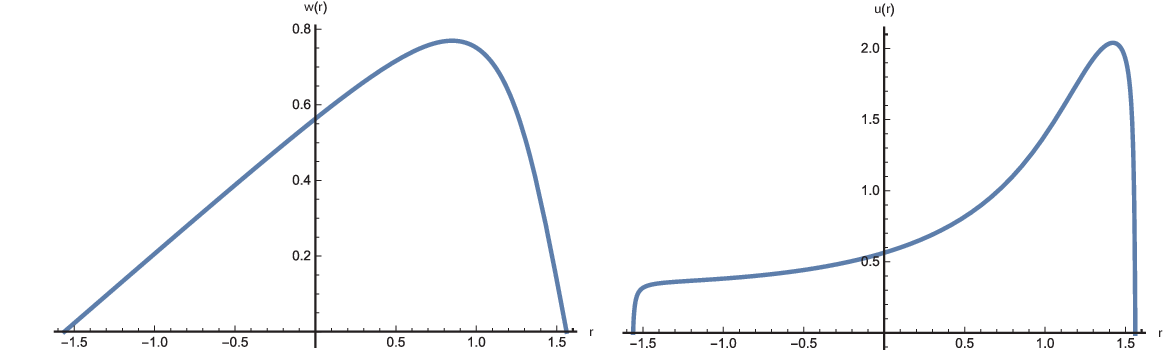}
\caption{The graphs of $r\mapsto w(r)$ and $r\mapsto u(r)$ for $r\in (0,a)$ when $\lambda=1$, $p=4.5$ and $a=\pi/2-0.01$. 
The value of the Rayleigh quotient is $R(w)=1.305$.}\label{Ex1}
 \end{center}
\end{figure}
\begin{figure}[htb]
 \begin{center}
  \includegraphics[scale=0.6]{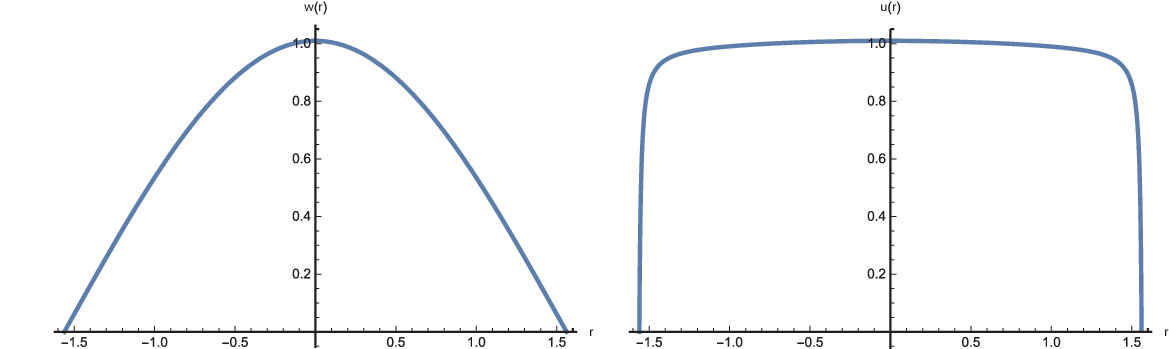}
\caption{The graphs of $r\mapsto w(r)$ and $r\mapsto u(r)$ for $r\in (0,a)$ when $\lambda=1$, $p=4.5$ and $a=\pi/2-0.01$. 
The value of the Rayleigh quotient is $R(w)=1.350$.}\label{Ex2}
 \end{center}
\end{figure}
\begin{figure}[htb]
 \begin{center}
  \includegraphics[scale=0.6]{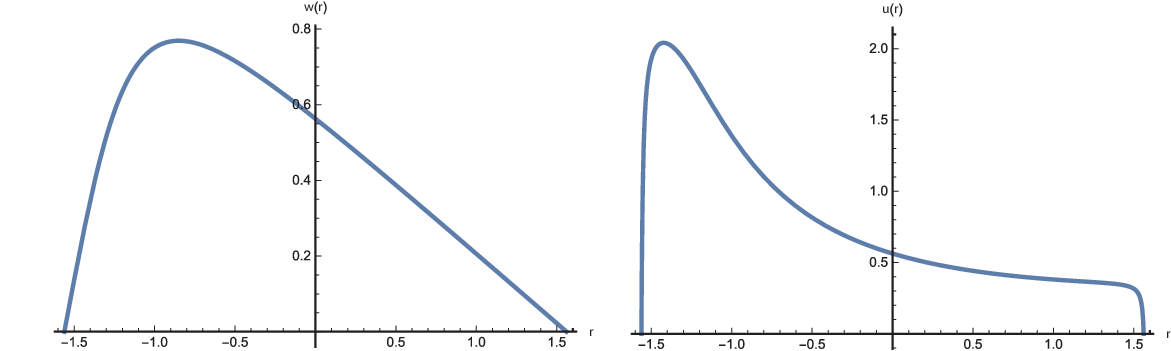}
\caption{The graphs of $r\mapsto w(r)$ and $r\mapsto u(r)$ for $r\in (0,a)$ when $\lambda=1$, $p=4.5$ and $a=\pi/2-0.01$. 
The value of the Rayleigh quotient is $R(w)=1.305$.}\label{Ex3}
 \end{center}
\end{figure}
Figures \ref{Ex1}, \ref{Ex2} and \ref{Ex3} are the results of $p=4.5>3/\lambda+1=4$ and $a=\pi/2-0.01$. Figures \ref{Ex1} and \ref{Ex3} represent non-even function solutions, while Figure \ref{Ex2} represents an even function solution. 
We can see that the Rayleigh quotient $R$ of non-even function solutions are
lower than that of the even function solution. Indeed the former has the value $R(w)=1.305$ while the latter is $R(w)=1.350$. 

\subsection{Multiple existence of even function solutions for $p>5$}
In this sub-section, we give a proof for a part of (iii) and (vi) of Theorem \ref{thmmain}.
\begin{lem}\label{Ea}
Let $\lambda>0$ and $p>5$. Then,
$E(b)\rightarrow 0$ as $b\rightarrow t(\pi/2)$, that is, $\varepsilon\rightarrow 0$, where $E(b)$ is the function defined by \eqref{Reyleigh}.
\end{lem}
\begin{proof}
Define $w_{\varepsilon}\in H^{1}_{0}(-b,b)$ as
\begin{align*}
w_{\varepsilon}(t)=
\begin{cases}
\sqrt{\varepsilon},& \vert t\vert \leq t(\frac{\pi}{2}-2\varepsilon)\\
\frac{\sqrt{\varepsilon}(b-|t|)}{b-t\left(\frac{\pi}{2}-2\varepsilon\right)},& t(\frac{\pi}{2}-2\varepsilon)\leq |t|\leq t(\frac{\pi}{2}-\varepsilon)=b.
\end{cases}
\end{align*}
Since $(t(r))_r=(\psi(r))^{-2}$,
the numerator part of the Rayleigh quotient $R$ is computed as
\begin{align*}
\int_{-b}^{b}\vert (w_{\varepsilon})_{t}\vert^{2}dt
&=2\int_{0}^{b}\vert (w_{\varepsilon})_{t}\vert^{2}dt \\
&=2\int_{t(\frac{\pi}{2}-2\varepsilon)}^{t(\frac{\pi}{2}-\varepsilon)}
\vert -\frac{\sqrt{\varepsilon}}{t\left(\frac{\pi}{2}-\varepsilon\right)-t\left(\frac{\pi}{2}-2\varepsilon\right)}\vert^{2}dt \\
&=2\frac{\varepsilon}{t(\frac{\pi}{2}-\varepsilon)-t(\frac{\pi}{2}-2\varepsilon)} 
\to 2\psi(\pi/2)^{2} \quad \mbox{as} \ \varepsilon\to0.
\end{align*}
Since $\lambda>0$, it holds that $\psi(r)=\cos\left(r\sqrt{1-\lambda}\right)\geq \cos\left(\frac{\pi}{2}\sqrt{1-\lambda}\right)=\psi(\pi/2)$ for $r\in[0,\pi/2]$.
By using the relation $\cos r\leq (\pi/2) -r$ on $[0,\pi/2]$,
the denominator part of $R$ is estimated as
\begin{align*}
\int_{-b}^{b}\left(\cos r(t)\right)^{1-p}\psi(r(t))^{p+3}\vert w_{\varepsilon}\vert^{p+1}dt
&=2\int_{0}^{b}\left(\cos r(t)\right)^{1-p}\psi(r(t))^{p+3}\vert w_{\varepsilon}\vert^{p+1}dt \\
&\geq 2\int_{0}^{t(\frac{\pi}{2}-2\varepsilon)} \left(\cos r(t)\right)^{1-p}\psi(r(t))^{p+3}\vert \sqrt{\varepsilon}\vert^{p+1}dt \\
&= 2\varepsilon^{\frac{p+1}{2}} \int_{0}^{\frac{\pi}{2}-2\varepsilon} \left(\cos r\right)^{1-p}\psi(r)^{p+3}(t(r))_r dr \\
&= 2\varepsilon^{\frac{p+1}{2}} \int_{0}^{\frac{\pi}{2}-2\varepsilon} \left(\cos r\right)^{1-p}\psi(r)^{p+1} dr \\
&\ge 2\psi\left(\frac{\pi}{2}\right)^{p+1} \varepsilon^{\frac{p+1}{2}} \int_{0}^{\frac{\pi}{2}-2\varepsilon} 
  \left(\frac{\pi}{2}-r\right)^{1-p}dr \\
&= \frac{2\psi\left(\frac{\pi}{2}\right)^{p+1} \varepsilon^{\frac{p+1}{2}}}{p-2} 
  \left( (2\varepsilon)^{2-p} -(\pi/2)^{2-p} \right) \\
&=\frac{2\psi\left(\frac{\pi}{2}\right)^{p+1}}{p-2}\left(2^{2-p}\varepsilon^{-\frac{p-5}{2}}-(\pi/2)^{2-p}\varepsilon^{\frac{p+1}{2}}\right)\rightarrow \infty\quad\text{ as }\varepsilon \rightarrow 0.
\end{align*} 
Thus, we obtain $R(w_{\varepsilon})\rightarrow 0$ as $\varepsilon\rightarrow 0$.

Let $W_{\varepsilon}$ be a minimizer of problem \eqref{Reyleigh}. From the above argument, we obtain
\begin{align*}
E(b)=R(W_{\varepsilon})\leq R(w_{\varepsilon})\rightarrow 0\quad \text{ as } \varepsilon\rightarrow 0. 
\end{align*}
\end{proof}

Now we give a proof of the multiple existence of the even function solutions for (iii) and (vi) of Theorem \ref{thmmain}.

\begin{proof}[Proof of the multiple existence of the even function solutions for \textup{(iii)} and \textup{(vi)} of Theorem \ref{thmmain}]
Assume $p>5$.
First we prove that, for each $b\in(0,t(\pi/2))$, there exists 
$\gamma_1>\lambda^{1/(p-1)}$ such that $Z(\gamma_1)=b$. 
We define
\begin{equation*}
 \alpha^* = \inf\{ \beta>0 \,|\, 
  0 < Z(\alpha) < t(\pi/2) \ \mbox{for} \ \alpha>\beta \}.
\end{equation*}
From Lemma \ref{zzero}, it follows that $0<Z(\alpha)<t(\pi/2)$ for all 
sufficiently large $\alpha>\lambda^{1/(p-1)}$.
Moreover, $Z(\lambda^{1/(p-1)})=t(\pi/2)$.
Hence, $\lambda^{1/(p-1)} \le \alpha^*<\infty$.
Since $Z(\alpha)$ is continuous at $\alpha$ for which $0<Z(\alpha)<t(\pi/2)$,
we see that $Z(\alpha)$ is continuous on $(\alpha^*,\infty)$.
From Lemma \ref{endpositive}, it follows that $Z(\alpha^*)$ exists in $(0,t(\pi/2)]$.
We prove $Z(\alpha^*)=t(\pi/2)$.
Suppose that $0<Z(\alpha^*)<t(\pi/2)$.
Then $Z(\alpha)$ is continuous at $\alpha=\alpha^*$,
and hence $0<Z(\alpha)<t(\pi/2)$ on $(\alpha^*-\delta,\alpha^*+\delta)$ for some 
$\delta>0$.
Since $0<Z(\alpha)<t(\pi/2)$ for $\alpha>\alpha^*$, we have 
$0<Z(\alpha)<t(\pi/2)$ for $\alpha>\alpha^*-\delta$.
This contradicts the definition of $\alpha^*$.
Therefore, $Z(\alpha^*)=\pi/2$.
From the continuous dependence of solutions on initial conditions,
it follows that $w(t;\alpha)$ converge to $w(t;\alpha^*)$ as 
$\alpha\to{\alpha^*}^+$ uniformly on each interval $[0,t_0]\subset[0,t(\pi/2))$.
Since $w(t;\alpha^*)>0$ for $0\le t <t(\pi/2)$, we see that
$Z(\alpha)\to t(\pi/2)$ as $\alpha\to{\alpha^*}^+$.
By Lemma \ref{zzero}, for each $b\in(0,t(\pi/2))$, there exists 
$\gamma_1>\alpha^*$ such that $Z(\gamma_1)=b$.

Let $W_{\varepsilon}$ be a least energy solution of the Rayleigh quotient \eqref{Reyleigh}.
Since $W_{\varepsilon}$ satisfies \eqref{eqw}, it holds
\begin{align*}
&\int_{-b}^{b}\left(\cos r(t)\right)^{1-p}\psi(r(t))^{p+3}
\left| W_{\varepsilon}\right|^{p+1}dt = \int_{-b}^{b}\left| \left(W_{\varepsilon}\right)_{t}\right|^{2}dt, 
\end{align*}
which implies
\begin{align*}
E(b)=R(W_{\varepsilon})=
\left(\int_{-b}^{b}\left| \left(W_{\varepsilon}\right)_{t}\right|^{2}dt \right)^{\frac{p-1}{p+1}}=
2^{\frac{p-1}{p+1}}
\left(\int_{0}^{b}\left| \left(W_{\varepsilon}\right)_{t}\right|^{2}dt \right)^{\frac{p-1}{p+1}}.
\end{align*}
Thus, from Lemma \ref{Ea}, we obtain
\begin{align*}
\int_{0}^{b}\left| \left(W_{\varepsilon}\right)_{t}\right|^{2}dt\rightarrow 0 \quad \text{ as } \varepsilon \rightarrow 0.
\end{align*}
Hence, it holds
\begin{align}\label{Wzero}
W_{\varepsilon}(0)=
-\int_{0}^{b}\left(W_{\varepsilon}\right)_{t}dt
=\int_{0}^{b}\left|\left(W_{\varepsilon}\right)_{t}\right|dt
\leq \sqrt{b}\left(\int_{0}^{b}\left|\left(W_{\varepsilon}\right)_{t}\right|^2 dt\right)^{\frac{1}{2}}
\rightarrow 0 \quad \text{ as } \varepsilon \rightarrow 0. 
\end{align}
Therefore, $0<W_{\varepsilon_1}(0)<\lambda^{1/(p-1)}/2$ for some $\varepsilon_1>0$.
We set $\alpha_1=W_{\varepsilon_1}(0)$ and define 
\begin{equation*}
 \alpha_* = \sup\{ \beta>0 \,|\, 
  0 < Z(\alpha) < t(\pi/2) \ \mbox{for} \ \alpha_1 \le \alpha<\beta \}.
\end{equation*}
Since $Z(\lambda^{1/(p-1)})=t(\pi/2)$, we have 
$\alpha_1<\alpha_*\le\lambda^{1/(p-1)}$.
By the same argument as above, we conclude that, 
for each $\varepsilon \in(0,\varepsilon_1)$,
there exists $\gamma_2 \in (\alpha_1,\alpha_*)$ such that $Z(\gamma_2)=b$.

By recalling \eqref{Wzero}, there exists $\varepsilon_2 \in (0,\varepsilon_1)$
such that if $\varepsilon\in(0,\varepsilon_2)$, then 
$W_\varepsilon(0)<\alpha_1$.
For each $\varepsilon \in(0,\varepsilon_2)$, 
we set $\gamma_3=W_\varepsilon(0)$. 
Then $Z(\gamma_3)=b$.

We have proven that if $\varepsilon \in(0,\varepsilon_2)$, 
then (5) with $b=t(\frac{\pi}{2}-\varepsilon)$ has three distinct even function 
solutions $w(t;\gamma_1)$, $w(t;\gamma_2)$ and $w(t;\gamma_3)$.
\end{proof}

 \begin{figure}[htb]
 \begin{center}
  \includegraphics[scale=0.6]{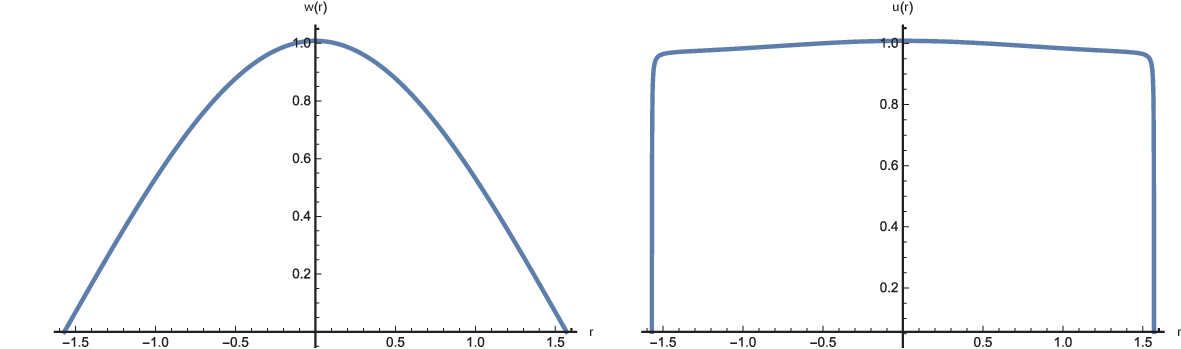}
\caption{The graphs of $r\mapsto w(r)$ and $r\mapsto u(r)$ for $r\in (0,a)$ when $\lambda=1$, $p=9$ and $a=1.57$. 
The value of the Rayleigh quotient is $R(w)=1.436$.}\label{Ex4}
 \end{center}
\end{figure}
\begin{figure}[H]
 \begin{center}
  \includegraphics[scale=0.6]{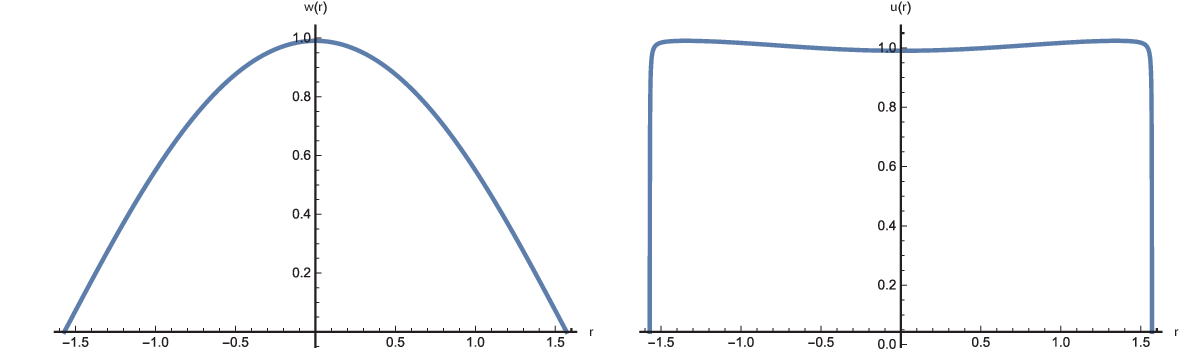}
  \caption{The graphs of $r\mapsto w(r)$ and $r\mapsto u(r)$ for $r\in (0,a)$ when $\lambda=1$, $p=9$ and $a=1.57$. 
The value of the Rayleigh quotient is $R(w)=1.437$.}\label{Ex5}
 \end{center}
\end{figure}
\begin{figure}[htb]
 \begin{center}
  \includegraphics[scale=0.6]{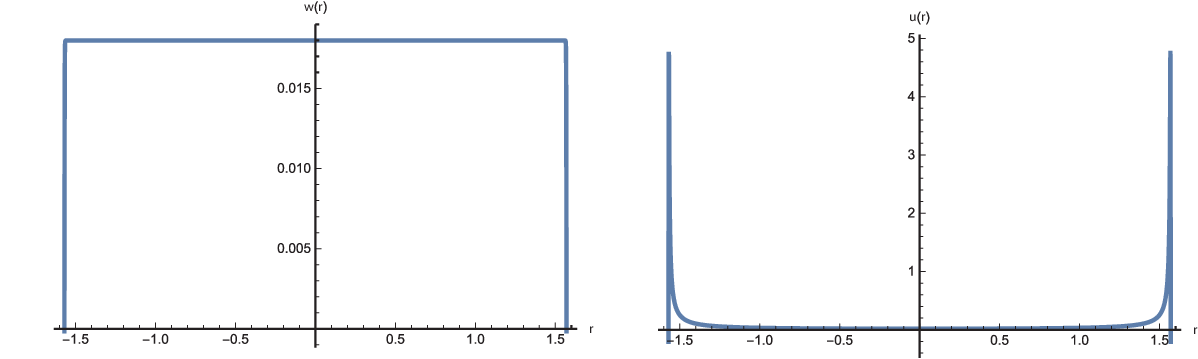}
\caption{The graphs of $r\mapsto w(r)$ and $r\mapsto u(r)$ for $r\in (0,a)$ when $\lambda=1$, $p=9$ and $a=1.57$. 
The value of the Rayleigh quotient is $R(w)=0.285$.}\label{Ex6}
 \end{center}
\end{figure}
We show the numerical experimental results for the case $\lambda=1$. 
Figures \ref{Ex4}, \ref{Ex5} and \ref{Ex6} are the results of $p=9$ and $a=1.57$. Figures \ref{Ex4}, \ref{Ex5} and \ref{Ex6} represent even function solutions.
We can see that the Rayleigh quotient $R$ of the solution of Figure \ref{Ex6} is smallest among three solutions. 
The solution of Figure \ref{Ex6} is the least energy solution in even function space.

\subsection{Multiple existence of non-even function solutions for $p>5$}
In this subsection, a part of assertion (iii) and (vi) of Theorem \ref{thmmain} is proved.
It seems to be difficult to verify this assertion in the same way as the case $3/\lambda+1<p\leq 5$, since the least energy solution in Figure \ref{Ex6} 
is not resembling the exact solution $w(t)=\lambda^{1/(p-1)}/\phi(r(t))$. 
So, we take another method, which investigates the sign of the second variation of $R$.
Let $W$ be a positive even function solution of \eqref{eqw} and 
\begin{align*}
\varphi(t,s):=\left(1+s f(t)\right)W(t), \quad -b<t<b,
\end{align*}
where
$\vert s \vert$ is small enough and $f$ is an odd function satisfying
\begin{align*}
 f(0)=0,\quad tf(t)>0 \ \mbox{for} \ t\neq 0.
\end{align*}
To compute the Rayleigh quotient, we define 
\begin{align*}
N(s)=\int_{-b}^{b}\varphi_{t}(t,s)^{2}dt, \quad 
D(s)=\left(\int_{-b}^{b}h(t)\varphi(t,s)^{p+1}dt\right)^{\frac{2}{p+1}}
\end{align*}
and set $R(s)=N(s)/D(s)$. 
Since $\varphi_{t}(t,s)=W_{t}(t)+s \left(f(t)W(t)\right)_{t}$, we obtain
\begin{align*}
\varphi_{t}(t,s)^{2}=W_{t}(t)^{2}+2sW_{t}(t)\left(f(t)W(t)\right)_{t}+s^{2}\left[\left(f(t)W(t)\right)_{t}\right]^2.
\end{align*}
Since $W_{t}(t)\left(f(t)W(t)\right)_{t}$ is an odd function, we have 
\begin{align}\label{N}
N(s)&=\int_{-b}^{b}W_{t}(t)^2dt+2s\int_{-b}^{b}W_{t}(t)\left(f(t)W(t)\right)_{t}dt+s^{2}\int_{-b}^{b}\left[\left(f(t)W(t)\right)_{t}\right]^2dt\\
&=\int_{-b}^{b}W_{t}(t)^2dt+s^{2}\int_{-b}^{b}\left[\left(f(t)W(t)\right)_{t}\right]^2dt\nonumber.
\end{align}
Hence, we obtain
\begin{align}\label{Nee}
N_{s}(0)=0,\quad N_{ss}(0)=2\int_{-b}^{b}\left[\left(f(t)W(t)\right)_{t}\right]^2dt.
\end{align}
For the denominator 
\begin{align*}
 D(s)=\left(\int_{-b}^{b}h(t)\left(1+s f(t)\right)^{p+1}W(t)^{p+1}dt\right)^{\frac{2}{p+1}},
\end{align*}
we have 
\begin{align}\label{De}
D_{s}(s)
&=\frac{2}{p+1}\left(\int_{-b}^{b}h(t)\left(1+s f(t)\right)^{p+1}W(t)^{p+1}dt\right)^{\frac{2}{p+1}-1}\cdot\\
&\quad\quad\quad\quad\left(\int_{-b}^{b}(p+1)h(t)\left(1+s f(t)\right)^{p}f(t)W(t)^{p+1}dt\right)\nonumber\\
&=2\left(\int_{-b}^{b}h(t)\left(1+s f(t)\right)^{p+1}W(t)^{p+1}dt\right)^{\frac{2}{p+1}-1}\cdot
\left(\int_{-b}^{b}h(t)\left(1+s f(t)\right)^{p}f(t)W(t)^{p+1}dt\right).\nonumber
\end{align}
Since $h(t)f(t)W(t)^{p+1}$ is an odd function, it holds $D_{s}(0)=0$.
Further, it holds
\begin{align*}
D_{ss}(s)=&2\left(\frac{2}{p+1}-1\right)
\left(\int_{-b}^{b}h(t)\left(1+s f(t)\right)^{p+1}W(t)^{p+1}dt\right)^{\frac{2}{p+1}-2}\cdot\\
&\quad\quad\quad\quad (p+1)\left(\int_{-b}^{b}h(t)\left(1+s f(t)\right)^{p}f(t)W(t)^{p+1}dt\right)^2\\
&+2\left(\int_{-b}^{b}h(t)\left(1+s f(t)\right)^{p+1}W(t)^{p+1}dt\right)^{\frac{2}{p+1}-1}\cdot\\
&\quad\quad\quad\quad p\left(\int_{-b}^{b}h(t)\left(1+s f(t)\right)^{p-1}f(t)^2W(t)^{p+1}dt\right),
\end{align*}
and hence
\begin{align}\label{Dee}
D_{ss}(0)=2p\left(\int_{-b}^{b}h(t)W(t)^{p+1}dt\right)^{\frac{2}{p+1}-1}\cdot\left(\int_{-b}^{b}h(t)f(t)^2W(t)^{p+1}dt\right).
\end{align}
Thus, we obtain
\begin{align}\label{Ree}
R_{s}(0)&=\frac{N_{s}(0)D(0)-N(0)D_{s}(0)}{D(0)^{2}}=0,\\
R_{ss}(0)&=\frac{N_{ss}(0)D(0)-N(0)D_{ss}(0)}{D(0)^{2}}-\frac{2D_{s}(0)}{D(0)^{3}}\left(N_{s}(0)D(0)-N(0)D_{s}(0)\right)\nonumber\\
&=\frac{N_{ss}(0)D(0)-N(0)D_{ss}(0)}{D(0)^{2}}.\nonumber
\end{align}
Putting 
\begin{align*}
I=\int_{-b}^{b}h(t)W(t)^{p+1}dt=\int_{-b}^{b}W_{t}(t)^{2}dt,
\end{align*}
we have
\begin{align}\label{secondv}
&N_{ss}(0)D(0)-N(0)D_{ss}(0)\\
&=2\left(\int_{-b}^{b}\left[\left(f(t)W(t)\right)_{t}\right]^{2}dt\right)\cdot I^{\frac{2}{p+1}}
-I\cdot 2pI^{\frac{2}{p+1}-1}\int_{-b}^{b}h(t)f(t)^{2}W(t)^{p+1}dt\nonumber\\
&=2I^{\frac{2}{p+1}}\left(\int_{-b}^{b}\left[\left(f(t)W(t)\right)_{t}\right]^{2}dt-
p\int_{-b}^{b}h(t)f(t)^{2}W(t)^{p+1}dt \right).\nonumber
\end{align}
Further, it holds that
\begin{align*}
&\int_{-b}^{b}\left[\left(f(t)W(t)\right)_{t}\right]^{2}dt\\
&=\left[\left(f(t)W(t)\right)\left(f(t)W(t)\right)_{t}\right]_{-b}^{b}-\int_{-b}^{b}\left(f(t)W(t)\right)\left(f(t)W(t)\right)_{tt}dt\\
&=-\int_{-b}^{b}f(t)f_{tt}(t)W(t)^2dt-2\int_{-b}^{b}f(t)f_{t}(t)W(t)W_{t}(t)dt-\int_{-b}^{b}f(t)^{2}W(t)W_{tt}(t)dt.
\end{align*}
The second term of above equality is rewritten as
\begin{align*}
2\int_{-b}^{b}f(t)f_{t}(t)W(t)W_{t}(t)dt 
& =\int_{-b}^{b}\left(f(t)^{2}\right)_{t}W(t)W_{t}(t)dt \\
&=-\int_{-b}^{b}f(t)^{2}\left(W(t)W_{t}(t)\right)_{t}dt\\
&=-\int_{-b}^{b}f(t)^{2}W_{t}(t)^{2}dt-\int_{-b}^{b}f(t)^{2}W(t)W_{tt}(t)dt.
\end{align*}
Hence, we obtain
\begin{align}\label{fphi}
\int_{-b}^{b}\left[\left(f(t)W(t)\right)_{t}\right]^{2}dt
=-\int_{-b}^{b}f(t)f_{tt}(t)W(t)^2dt+\int_{-b}^{b}f(t)^{2}W_{t}(t)^{2}dt.
\end{align}
Substituting \eqref{fphi} to \eqref{secondv}, we get
\begin{align*}
N_{ss}(0)D(0)-N(0)D_{ss}(0)=2I^{\frac{2}{p+1}}F(b),
\end{align*}
where
\begin{align}\label{F}
F(b)=-\int_{-b}^{b}f(t)f_{tt}(t)W(t)^{2}dt+\int_{-b}^{b}f(t)^{2}W_{t}(t)^{2}dt-
p\int_{-b}^{b}h(t)f(t)^{2}W(t)^{p+1}dt.
\end{align}
From now on, we fix $f(t)=t$. Then, it holds:
\begin{align}\label{Fa}
F(b)&= \int_{-b}^{b}t^{2}W_{t}(t)^{2}dt-p\int_{-b}^{b}t^{2}h(t)W(t)^{p+1}dt\\
&\leq b^{2}\int_{-b}^{b}W_{t}(t)^{2}dt-p\int_{-b}^{b}t^{2}h(t)W(t)^{p+1}dt\nonumber\\
&=b^{2}\int_{-b}^{b}h(t)W(t)^{p+1}dt-p\int_{-b}^{b}t^{2}h(t)W(t)^{p+1}dt\nonumber\\
&=\int_{-b}^{b}\left(b^{2}-pt^{2}\right)h(t)W(t)^{p+1}dt\nonumber\\
&=2\int_{0}^{b}\left(b^{2}-pt^{2}\right)h(t)W(t)^{p+1}dt.\nonumber
\end{align}

Now, we give a proof for the multiple existence of non-even function solutions of assertion (iii) and (vi) of Theorem \ref{thmmain} for the case $p>5$.
\begin{proof}
[Proof of multiple existence of non-even function solutions of \textup{(iii)} and \textup{(vi)} of Theorem \ref{thmmain}]
Let $p>5$.
First we claim that 
\begin{equation*}
 \lim_{s\to t(\pi/2)} \int_{0}^{s} 
   \left(t\Bigl(\dfrac{\pi}{2}\Bigr)^{2}-pt^{2}\right)h(t)dt = -\infty.
\end{equation*}
Let $s\in(0,t(\pi/2))$ be sufficiently close to $t(\pi/2)$ for which 
$t(\pi/2)/2<s<t(\pi/2)$.
Since
\begin{multline*}
 \int_{0}^{s} 
   \left(t\Bigl(\dfrac{\pi}{2}\Bigr)^{2}-pt^{2}\right)h(t)dt \\
 = \int_{0}^{t(\pi/2)/2} 
   \left(t\Bigl(\dfrac{\pi}{2}\Bigr)^{2}-pt^{2}\right)h(t)dt
 + \int_{t(\pi/2)/2}^{s} 
   \left(t\Bigl(\dfrac{\pi}{2}\Bigr)^{2}-pt^{2}\right)h(t)dt,
\end{multline*}
it is sufficient to prove that
\begin{equation}
 \lim_{s\to\pi/2} \int_{t(\pi)/2}^{s} 
   \left(t\Bigl(\dfrac{\pi}{2}\Bigr)^{2}-pt^{2}\right)h(t)dt = -\infty.
 \label{limint}
\end{equation}
We observe that
\begin{align*}
 \int_{t(\pi/2)/2}^{s}
  \left(t\Bigl(\dfrac{\pi}{2}\Bigr)^{2}-pt^{2}\right)h(t)dt
 & \le \left(t\Bigl(\dfrac{\pi}{2}\Bigr)^{2}-\frac{p}{4}t\Bigl(\dfrac{\pi}{2}\Bigr)^{2}\right) \int_{t(\pi/2)/2}^{s} (\cos r(t))^{1-p}\psi(r(t))^{p+3} dt \\
 & = -\frac{p-4}{4} t\Bigl(\dfrac{\pi}{2}\Bigr)^{2} 
   \int_{t(\pi/2)/2}^{s} (\cos r(t))^{1-p}\psi(r(t))^{p+3} dt \\
 & \le -\frac{p-4}{4} t\Bigl(\dfrac{\pi}{2}\Bigr)^{2} \psi(r(s))^{p+1}
   \int_{t(\pi/2)/2}^{s} (\cos r(t))^{1-p}\psi(r(t))^{2} dt \\
 & = -\frac{p-4}{4} t\Bigl(\dfrac{\pi}{2}\Bigr)^{2} \psi(r(s))^{p+1}
   \int_{r(t(\pi/2)/2)}^{r(s)} (\cos r)^{1-p} dr \\
 & \le -\frac{p-4}{4} t\Bigl(\dfrac{\pi}{2}\Bigr)^{2} \psi(\pi/2)^{p+1}
   \int_{r(t(\pi/2)/2)}^{r(s)} (\cos r)^{1-p} \sin r dr \\
 & = -\frac{p-4}{4} t\Bigl(\dfrac{\pi}{2}\Bigr)^{2} \psi(\pi/2)^{p+1}
      \frac{1}{p-2} \left( (\cos r(s))^{2-p} - C \right)
\end{align*}
for some constant $C>0$, which means that \eqref{limint} holds.
Consequently, we can take $c\in(t(\pi/2)/2,t(\pi/2))$ such that 
\begin{equation*} 
 \int_{0}^{c} 
   \left(t\Bigl(\dfrac{\pi}{2}\Bigr)^{2}-pt^{2}\right)h(t)dt \le -1.
\end{equation*}

Set $W=W_{\varepsilon}$, where $W_\varepsilon$ is a least energy 
solution of the Rayleigh quotient \eqref{Reyleigh} as in the proof of 
the multiple existence of the even function solutions.
Next we will prove that  $W(t)/W(0)$ converges to $1$ uniformly on 
$[0,c]$ as $\varepsilon\to0$. 
We note that $c<b<t(\pi/2)$ if $\varepsilon>0$ is sufficiently small.
We have
\begin{equation*}
 W(t) - W(0) = -\int_0^t \int_0^\tau h(\sigma) W(\sigma)^p d\sigma d\tau,
 \quad t \in [0,c].
\end{equation*}
Since $W(t)$ is positive and decreasing on $[0,b)$, we observe that
\begin{align*}
 \left| \frac{W(t)}{W(0)} - 1 \right|
  & \le W(0)^{p-1} \max_{\sigma\in[0,c]} h(\sigma) 
        \int_0^t \int_0^\tau d\sigma d\tau \\
  & \le \frac{1}{2}c^2 W(0)^{p-1} \max_{\sigma\in[0,c]} h(\sigma), 
    \quad t \in [0,c].
\end{align*}
Therefore, $W(t)/W(0)$ converges to $1$ uniformly on 
$[0,c]$ as $\varepsilon\to0$ (recall that $W(0)=W_{\varepsilon}(0)\rightarrow 0$ as $\varepsilon\rightarrow 0$ form \eqref{Wzero}).

Let $c \in (t(\pi/2)/2,t(\pi/2))$.
We note that $t(\pi/2)^2-pt^2<0$ for $c\le t \le t(\pi/2)$.
Let $\varepsilon>0$ be sufficiently small for which $c<b<t(\pi/2)$.
From \eqref{Fa} it follows that  
\begin{align*}
 F(b) & \le 2 \int_{0}^{b}\left(t\Bigl(\dfrac{\pi}{2}\Bigr)^{2}-pt^2\right)h(t)W(t)^{p+1}dt \\
      & \le 2 \int_{0}^{c}\left(t\Bigl(\dfrac{\pi}{2}\Bigr)^{2}-pt^2\right)h(t)W(t)^{p+1}dt \\
 & = 2 W(0)^{p+1} \int_{0}^{c}\left(t\Bigl(\dfrac{\pi}{2}\Bigr)^{2}-pt^2\right)h(t)
                                    \left(\frac{W(t)}{W(0)}\right)^{p+1}dt,
\end{align*}
Since
\begin{equation*}
 \lim_{\varepsilon\to0^+} 
  \int_{0}^{c}\left(t\Bigl(\dfrac{\pi}{2}\Bigr)^{2}-pt^2\right)h(t)
                                    \left(\frac{W(t)}{W(0)}\right)^{p+1}dt
 = \int_{0}^{c}\left(t\Bigl(\dfrac{\pi}{2}\Bigr)^{2}-pt^2\right)h(t) dt \le -1,
\end{equation*}
we conclude that $F(b)<0$ for all sufficiently small $\varepsilon>0$.
Thus $W$ is not a minimizer of \eqref{Reyleigh-ne}
and the infimum of \eqref{Reyleigh-ne} is attained by non-even functions.
\end{proof}

 \begin{figure}[htb]
 \begin{center}
  \includegraphics[scale=0.6]{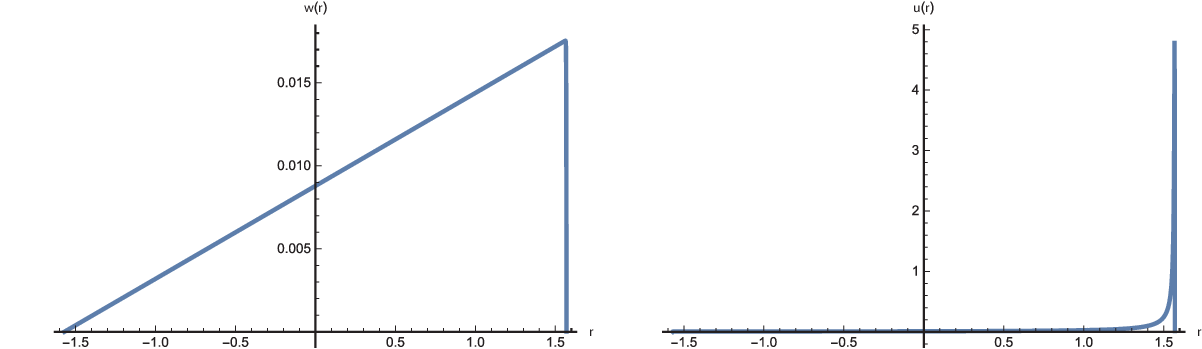}
\caption{The graphs of $r\mapsto w(r)$ and $r\mapsto u(r)$ for $r\in (0,a)$ when $\lambda=1$, $p=9$ and $a=1.57$. 
The value of the Rayleigh quotient is $R(w)=0.166$.}\label{Ex7}
 \end{center}
\end{figure}
\begin{figure}[htb]
 \begin{center}
  \includegraphics[scale=0.6]{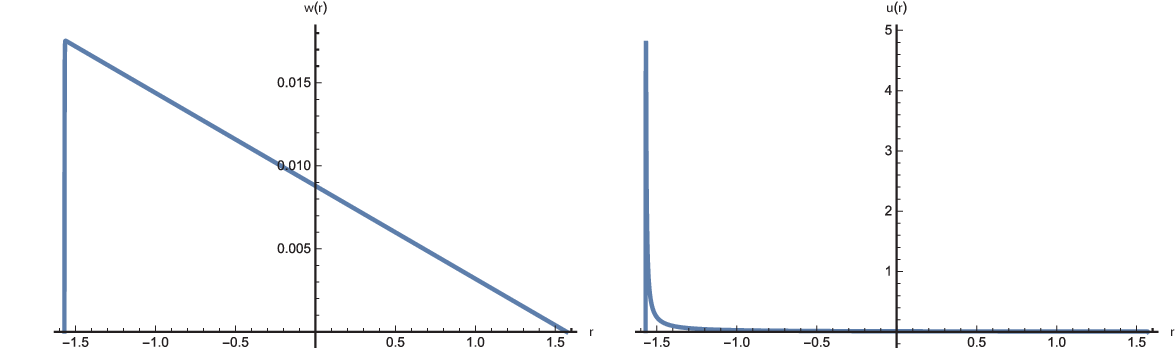}
\caption{The graphs of $r\mapsto w(r)$ and $r\mapsto u(r)$ for $r\in (0,a)$ when $\lambda=1$, $p=9$ and $a=1.57$. 
The value of the Rayleigh quotient is $R(w)=0.166$.}\label{Ex8}
 \end{center}
\end{figure}
We show the numerical experimental results for the case $\lambda=1$. Figures \ref{Ex7} and \ref{Ex8} are the results of $p=9$ and $a=1.57$. There parameters are the same as Figures \ref{Ex4}, \ref{Ex5} and \ref{Ex6}.
Figures \ref{Ex7} and \ref{Ex8} represent non-even function solutions.
As in the proof of multiple existence of non-even function solutions for (iii) and (vi) of Theorem \ref{thmmain} for $p>5$, the value of the Rayleigh quotient $R$ ($=0.166$) for these solutions are smaller than 
that of the least energy solution of the even-function space ($=0.285$, Figure \ref{Ex6}). 

\section{Appendix A (Proof of Lemma \ref{zzero})}

\begin{proof}[Proof of Lemma \ref{zzero}]
Let $\delta: 0<\delta< t(\pi/4)$ be arbitrary and
let $\alpha$ satisfy $\alpha^{p-1}>(p+2)\left(\cos\pi\sqrt{1-\lambda}/4\right)^{-p-3}\delta^{-2}$.
It is sufficient to prove that $w(t;\alpha)$ has a zero in $[0,\delta]$.
Assume to the contrary that $w(t;\alpha)>0$ for $0 \le t \le \delta$.
Then, since $w_{tt}(t;\alpha)<0$ for $t\in[0,\delta]$, we see that
$w(t;\alpha)$ is concave on $[0,\delta]$, which implies that
\begin{equation*}
 w(t;\alpha) \ge \frac{\alpha}{\delta}(\delta-t), \quad 0 \le t \le \delta.
\end{equation*}
Integrating the differential equation in \eqref{initw} on $[0,s]$ and 
integrating it on $[0,\delta]$ again, we get
\begin{equation*}
  \alpha \ge -w(\delta;\alpha) + \alpha 
  = \int_0^\delta \int_0^s h(t) w(t;\alpha)^p dt ds.
\end{equation*}
Hence, since $h(t)=\left(\cos r(t)\right)^{1-p}\left(\cos\sqrt{1-\lambda}r(t)\right)^{p+3}\geq
\left(\cos\pi\sqrt{1-\lambda}/4\right)^{p+3}=:
h_0>0$ on $(0,\delta)$, we observe that
\begin{align*}
 \alpha \ge h_0 \int_0^\delta \int_0^s
   \left(\frac{\alpha}{\delta}(\delta-t)\right)^p dt ds
  = h_0 \frac{\alpha^p \delta^2}{p+2},
\end{align*}
which means that $\alpha^{p-1}\le(p+2)h_{0}^{-1}\delta^{-2}$.
This is a contradiction.
Therefore, $w(t;\alpha)$ has a zero in $[0,\delta]$. 
\end{proof}

\section{Appendix B Proof of Theorem \ref{thmmain}-(iv)}
As in the proof of claim (iii) of Theorem \ref{thmmain}, by Theorem 13-(i)-(b) of \cite{MR3470747}, if
$p\leq \min\{6/\lambda-3,5 \}=6/\lambda$, then \eqref{eqw} has a unique positive radial solution.
Thus, we consider the case 
\begin{align}\label{SWass}
p>\frac{6}{\lambda}-3\quad\quad \text{ and }\quad \quad \frac{3}{4}<\lambda\leq 1.
\end{align}
We will assume $\lambda<1$ for a while.
We show that $\varphi(r,a)$ in \eqref{genuniq} has at most one critical point (a minimum point if it exists) on $(0,a)$. 
We note $\varphi(r,a)$ in \eqref{genuniq} is expressed as:
\begin{align}\label{genuniq2}
2\varphi(r,a):=&-2(p-1)(\sin r)+\left(\sec a\sqrt{1-\lambda}\right)^{2}\cdot\\
&\cdot\Bigl(
(p-1)\left[\left(\cos 2r\sqrt{1-\lambda}\right)+1\right]\left(\sin r\right)-(p+3)\sqrt{1-\lambda}(\sin 2r\sqrt{1-\lambda})\left(\cos r\right)
\Bigr).\nonumber
\end{align}
Derivating $2\varphi(r,a)$ in \eqref{genuniq2} with respect to $r$, we obtain
\begin{align*}
2\varphi_{r}(r,a)=&-2(p-1)(\cos r)\\
&+\left(\sec a \sqrt{1-\lambda }\right)^{2} \Bigl[(\cos r)\left(
p-1+\left(-7-p+2(p+3)\lambda\right)\left(\cos 2r\sqrt{1-\lambda}\right)
\right)\\
&+(5-p)\sqrt{1-\lambda}(\sin r)\left(\sin 2r\sqrt{1-\lambda}\right) \Bigr]\\
=&\left(\sec a \sqrt{1-\lambda }\right)^{2}(\cos r)\cdot\Bigl\{
-2(p-1)\left(\cos a \sqrt{1-\lambda }\right)^{2}\\
&+\Bigl[
p-1+\left(-7-p+2(p+3)\lambda\right)\left(\cos 2r\sqrt{1-\lambda}\right)\\
&+(5-p)\sqrt{1-\lambda}(\tan r)\left(\sin 2r\sqrt{1-\lambda}\right) \Bigr]
\Bigr\}.
\end{align*} 
Define 
\begin{align}
\phi(r)=&
p-1+\left(-7-p+2(p+3)\lambda\right)\left(\cos 2r\sqrt{1-\lambda}\right)\\
&+(5-p)\sqrt{1-\lambda}(\tan r)\left(\sin 2r\sqrt{1-\lambda}\right).\nonumber
\end{align}
\begin{figure}[htb]
 \begin{center}
  \includegraphics[scale=1.0]{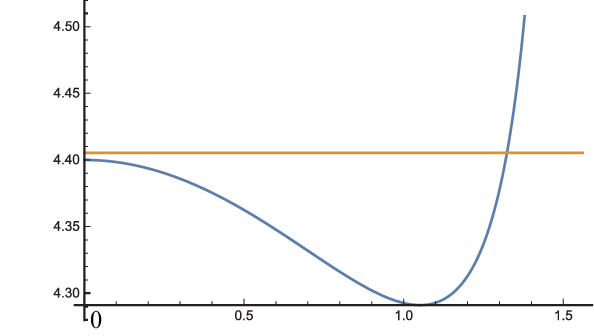}
\caption{The graphs of $r\mapsto \phi(r)$ and $r\mapsto -2(p-1)\left(\cos a \sqrt{1-\lambda }\right)^{2}$ for $r\in (0,a)$.}\label{fth3}
 \end{center}
\end{figure}
\noindent
Figure \ref{fth3} represents the graphs of $r\mapsto \phi(r)$ and $r\mapsto 2(p-1)\left(\cos a \sqrt{1-\lambda }\right)^{2}$ (horizontal line) for $r\in (0,a)$. We will show that two graphs have at most one intersection point on $(0,a)$ (if $a$ is sufficiently close to $\pi/2$ two graphs have exactly one intersection point on $(0,a)$).  
We investigate the behavior of graph of $r\mapsto \phi(r)$.
Derivating $\phi$, we obtain
\begin{align}
\phi_{r}(r)=&\sqrt{1-\lambda}\left[
2\left(p+7-2(p+3)\lambda\right)+(5-p)(\sec r)^{2}
\right]\left(\sin 2r\sqrt{1-\lambda}\right)\\
&+2(5-p)(1-\lambda)\left(\cos 2r\sqrt{1-\lambda}\right)\tan r.\nonumber
\end{align}
Next we show that $\phi_{r}$ has a unique zero point on $(0,\pi/2)$.
To see this we observe that 
\begin{align*}
r\mapsto \sqrt{1-\lambda}\left[
2\left(p+7-2(p+3)\lambda\right)+(5-p)(\sec r)^{2}
\right]=:\phi_{1}(r)
\end{align*} 
and
\begin{align*}
r\mapsto -2(5-p)(1-\lambda)\left(\cot 2r\sqrt{1-\lambda}\right)\tan r=:\phi_{2}(r)
\end{align*}
has a unique intersection point on $(0,\pi/2)$. 
\begin{figure}[htb]
\begin{center}
  \includegraphics[scale=0.7]{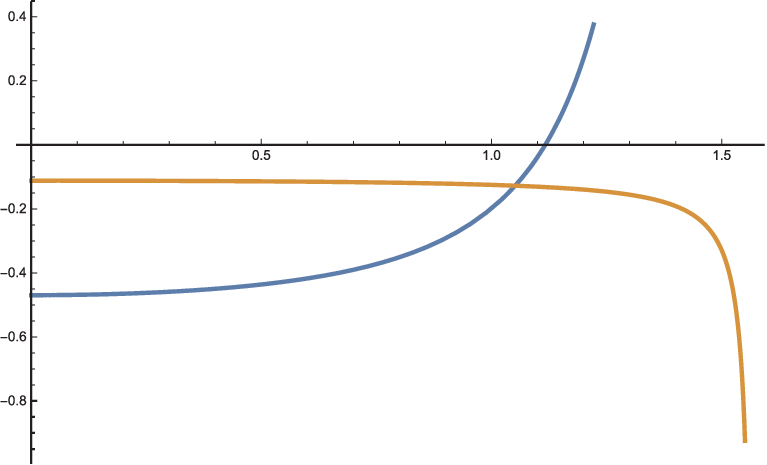}
\caption{The graphs of $r\mapsto \phi_{1}(r)$ (incleasing curve) and $r\mapsto \phi_{2}(r)$ (decreasing curve) for $r\in (0,\pi/2)$.}\label{fth3-2}
 \end{center}
\end{figure}
Since $3/4<\lambda<1$,
\begin{align*}
\phi_{1;r}(r)=&2(5-p)\sqrt{1-\lambda}(\sec r)^{2}\tan r>0\\
\phi_{2;r}(r)=&-\frac{(5-p)(1-\lambda)}{(\sin 2r\sqrt{1-\lambda})^{2}(\cos r)^{2}}\left(
\sin 4r\sqrt{1-\lambda}-2\sqrt{1-\lambda}\sin 2r\right)<0,
\end{align*}
$\phi_{1}$ is monotone increasing and $\phi_{2}$ is monotone decreasing on $(0,\pi/2)$; see Figure \ref{fth3-2}.
Moreover, since $3/4<\lambda<1$, it holds
\begin{align*}
\lim_{r\rightarrow \pi/2}\phi_{1}(r)=\infty \quad \text{ and }\quad
\lim_{r\rightarrow \pi/2}\phi_{2}(r)=-\infty
\end{align*}
and since $p>6/\lambda -3$,
\begin{align}\label{zeroineq}
\phi_{1}(0) < \lim_{r\rightarrow 0}\phi_{2}(r).
\end{align}
We note from \eqref{SWass}, \eqref{zeroineq} holds.
Thus, we have shown that $\phi_{r}$ has a unique zero point on $(0,\pi/2)$, and hence $\phi$ has a unique critical point (minimum point) on $(0,\pi/2)$.
We show that $2(p-1)\left(\cos a\sqrt{1-\lambda}\right)^{2}>\phi(0)=-8+2(p+3)\lambda$. 
Since $2(p-1)\left(\cos a\sqrt{1-\lambda}\right)^{2}\geq 2(p-1)\left(\cos (\pi/2)\sqrt{1-\lambda}\right)^{2}$
it is sufficient to show
\begin{align*}
2(p-1)\left(\cos \frac{\pi}{2}\sqrt{1-\lambda}\right)^{2}\geq -8+2(p+3)\lambda .
\end{align*}
This is equivalent to
\begin{align*}
\left(2\lambda-1-\cos\pi\sqrt{1-\lambda}\right)p\leq 7-6\lambda -\cos\pi\sqrt{1-\lambda}.
\end{align*}
We can see that $\left(2\lambda-1-\cos\pi\sqrt{1-\lambda}\right)>0$ for $3/4<\lambda<1$. Thus, if
\begin{align*}
 p\leq \left(7-6\lambda-\cos\pi\sqrt{1-\lambda}\right)/\left(2\lambda-1-\cos\pi\sqrt{1-\lambda}\right)=I(\lambda),
\end{align*}
then $\varphi(r,a)$ has at most one criptical point (minimum point, if it exists) on $(0,a)$.
In addition, it holds
\begin{align*}
\varphi(0,a)=&0\\
\varphi(a,a)=&-(p+3)\sqrt{1-\lambda}(\cos a)(\tan a\sqrt{1-\lambda})<0,
\end{align*}
so $\varphi(r,a)<0$ on $(0,a)$. This completes the proof for the case $3/4<\lambda<1$.

Next, we show the case $\lambda=1$. In this case, since $t=r$, we obtain 
\begin{align}\label{Hrep3}
H(r)&=-b(r)h(r)+\frac{1}{p+1}\left(a(r)h(r)\right)_{r}\\
&=-\frac{p+3}{p+1}rh(r)+\frac{1}{p+1}\left(a^2-r^2\right)h_{r}(r)\nonumber\\
&=-\frac{p+3}{p+1}r\left(\cos r\right)^{1-p}+\frac{p-1}{p+1}\left(a^2-r^2\right)\left(\cos r\right)^{-p}\left(\sin r\right)\nonumber\\
&=\frac{p+3}{p+1}\left(\cos r\right)^{-p}(\sin r)\left[
-r(\cot r)+\frac{p-1}{p+3}(a^2-r^2)\right]\nonumber\\
&\leq \frac{p+3}{p+1}\left(\cos r\right)^{-p}(\sin r)\left[
-r(\cot r)+\frac{p-1}{p+3}\left(\left(\frac{\pi}{2}\right)^2-r^2\right)\right]\nonumber
\end{align} 
We put
\begin{align*}
 \varphi(r):= -r(\cot r)+\frac{p-1}{p+3}\left(\frac{\pi}{2}^2-r^2\right).
\end{align*}
By direct computation, we obtain
\begin{align*}
\varphi_{rrr}(r)=\frac{1}{\left(\sin r\right)^{4}}\left(2r(2+\cos 2r)-3\sin 2r\right)>0,\,\,\,\,r\in (0,\pi/2),
\end{align*}
since 
\begin{align*}
\left(2r(2+\cos 2r)-3\sin 2r\right)_{r}=4(\sin 2r)(-r+\tan r)>0,\,\,\,\,r\in (0,\pi/2)
\end{align*}
and
\begin{align*}
\left(2r(2+\cos 2r)-3\sin 2r\right)\Big|_{r=0}=0.
\end{align*}
Hence $\varphi_{rr}$ is monotone increasing. 
Further, we have
\begin{align*}
\lim_{r\rightarrow 0}\varphi_{rr}(r)=\lim_{r\rightarrow 0}-\frac{2 \left((p+3) (r \cot (r)-1) \csc ^2(r)+p-1\right)}{p+3}=-\frac{4(p-3)}{3(p+3)}
\end{align*}
and $\varphi_{rr}\left(\frac{\pi}{2}\right)=8/(p+3)>0$.
Thus, in the case $1<p\leq 3$, $\varphi_{r}$ is monotone increasing on $(0,\pi/2)$, while in the case $3<p$, $\varphi_{r}$ has one minimum point on $(0,\pi/2)$.
Hence, from
\begin{align*}
\lim_{r\rightarrow 0}\varphi_{r}(r)=\lim_{r\rightarrow 0}\left(r \left(\frac{8}{p+3}+\csc ^2(r)-2\right)-\cot (r)\right)=0
\end{align*}
and $\varphi_{r}\left(\frac{\pi}{2}\right)=\frac{\pi}{2}\cdot\frac{5-p}{p+3}>0$,
we see that in the case $1<p\leq 3$, $\varphi$ is monotone increasing  on $(0,\pi/2)$, while in the case $3<p$, $\varphi$
has one minimal point on $(0,\pi/2)$.
Since 
\begin{align*}
\varphi(0)=\frac{p-1}{p+3}a^{2}-1<\frac{p-1}{p+3}\left(\frac{\pi}{2}\right)^{2}-1\leq 0
\end{align*}
We note the last inequality can be rewritten as
\begin{align*}
p\leq \frac{\pi^2+12}{\pi^2-4}. 
\end{align*}
In addition, $\varphi(\pi/2)=0$ holds, thus if $1<p\leq (\pi^2+12)/(\pi^2-4)=\lim_{\lambda\rightarrow 1}I(\lambda)$, we obtain the assertion $H(r)<0$, $r\in (0,a)$. This completes the proof.
\qed

\noindent
\begin{large}
{\bf Acknowledgement}\\
\end{large}
This work is partially supported by 
JSPS KAKENHI Grant Number 18K03387, 19K03595,
from Japan Society for the Promotion of Science.

\end{document}